\documentclass[10pt]{article}

\usepackage{bm}
\usepackage[top=0.9in,bottom=1.1in,left=1.05in,right=1.05in]{geometry}

\usepackage{textgreek,mathtools}
\usepackage{amsmath}
\usepackage{graphicx}
\usepackage[colorinlistoftodos]{todonotes}
\usepackage[colorlinks=true, allcolors=blue]{hyperref}
\usepackage{amsmath}
\usepackage{upgreek}
\usepackage{amssymb}
\usepackage{amsfonts}
\usepackage{amsthm}
\usepackage{mathrsfs}
\usepackage{fontenc}
\usepackage{empheq}
\usepackage{enumerate}
\usepackage{comment}
\usepackage{appendix}
\usepackage{bbm}
\usepackage{mathtools}
\mathtoolsset{showonlyrefs}
\usepackage[stable]{footmisc}

\theoremstyle{plain}
\newtheorem{thm}{Theorem}[section]

\newtheorem{lem}[thm]{Lemma}
\newtheorem{prop}[thm]{Proposition}

\newtheorem{assu}[thm]{Assumption}

\theoremstyle{definition}
\newtheorem{defn}[thm]{Definition}

\theoremstyle{remark}
\newtheorem{rem}[thm]{Remark}

\newcommand{\bX}{\bm{X}}
\newcommand{\bx}{\bm{x}}

\newcommand{\balpha}{\bm{\alpha}}

\newcommand{\RR}{\mathbb{R}}
\newcommand{\EE}{\mathbb{E}}

\newcommand{\PP}{\mathbb{P}}
\newcommand{\mbA}{\mathbb{A}}

\newcommand{\MCL}{\mathcal{L}}
\newcommand{\MCP}{\mathcal{P}}
\newcommand{\ud}{\,\mathrm{d}}
\newcommand{\mc}[1]{\mathcal{#1}}
\newcommand{\half}{\frac{1}{2}}

\newcommand{\eps}{\epsilon}
\newcommand{\abs}[1]{\left|#1\right|}
\newcommand{\average}[1]{\left\langle#1\right\rangle}

\newcommand{\dF}{\mathrm{D}_{{\Phi}}}
\newcommand{\barXti}{{\bar X_t^i}}
\newcommand{\barXTi}{{\bar X_T^i}}
\newcommand{\barXsi}{{\bar X_s^i}}

\newcommand{\bR}{\mathbb {R}}
\newcommand{\transpose}{^{\operatorname{T}}}
\newcommand{\rmd}{\,\mathrm{d}}
\newcommand{\rmD}{\mathrm{D}}

\begin{document}

\title{A class of dimension-free metrics for the convergence of empirical measures}



\author{Jiequn Han\thanks{Department of Mathematics, Princeton University, Princeton, NJ 08544-1000, USA, \em{jiequnh@princeton.edu }.} \and Ruimeng Hu\thanks{Department of Mathematics, and Department of Statistics and Applied Probability, University of California, Santa Barbara, CA 93106-3080, {\em rhu@ucsb.edu}. RH was partially supported by the NSF grant DMS-1953035, and the Faculty Career Development Award, the Research Assistance Program Award, the Early Career Faculty Acceleration funding and the Regents' Junior Faculty Fellowship at the University of California, Santa Barbara.} \and  Jihao Long\thanks{The Program in Applied and Computational Mathematics, Princeton University, Princeton, NJ 08544-1000, \em{jihaol@princeton.edu}.}}
\date{\today}
\maketitle

\begin{abstract}
    This paper concerns the convergence of empirical measures in high dimensions. We propose a new class of probability metrics and show that under such metrics, the convergence is free of the curse of dimensionality (CoD). Such a feature is critical for high-dimensional analysis and stands in contrast to classical metrics ({\it e.g.}, the Wasserstein metric).
    The proposed metrics fall into the category of integral probability metrics, for which we specify criteria of test function spaces to guarantee the property of being free of CoD. Examples of the selected test function spaces include the reproducing kernel Hilbert spaces, Barron space, and flow-induced function spaces. Three applications of the proposed metrics are presented: 1. The convergence of empirical measure in the case of random variables; 2. The convergence of $n$-particle system to the solution to McKean-Vlasov stochastic differential equation; 3. The construction of an $\varepsilon$-Nash equilibrium for a homogeneous $n$-player game by its mean-field limit. As a byproduct, we prove that, given a distribution close to the target distribution measured by our metric and a certain representation of the target distribution, we can generate a distribution close to the target one in terms of the Wasserstein metric and relative entropy. Overall, we show that the proposed class of metrics is a powerful tool to analyze the convergence of empirical measures in high dimensions without CoD. 
\end{abstract}


\section{Introduction}\label{sec:intro}

The convergence of empirical measures plays a crucial role in analyzing the efficiency of mean-field theory or mean-field games (MFGs), which are fundamental tools to approximate finite-particle or finite-agent systems in the asymptotic of a very large population. Specifically, mean-field theory studies the behavior of a high-dimensional stochastic particle system by considering the effect of all other particles approximated by an average single effect. It plays a significant role in many fields, for instance, statistical physics \cite{bossy2005some,chaintron2021propagation}. MFGs were introduced independently by Lasry-Lions (\cite{LaLi1:2006,LaLi2:2006,LaLi:2007}) and Huang-Malham\'{e}-Caines  (\cite{HuMaCa:06,HuCaMa:07}). 
MFGs study the decision-making problem of a continuum of agents, and are able to provide approximations to Nash equilibria of $n$-player games in which players interact through their empirical measure. For further background on MFGs, we refer to the books \cite{CaDe1:17,CaDe2:17} and the references therein.

In stochastic analysis,  there is a rich literature on the convergence analysis of $n$ interacting-bodies/particle system to the corresponding limit (also known as the McKean-Vlasov system \cite{mckean1966class,mckean1967propagation}). Recent developments in this field can be found in \cite{lacker2018strong,lacker2018mean,jabir2019rate}. In general, the distance between an $n$-body empirical measure and its limit is of order $n^{-c/d}$, where $d$ is the dimension of one body and $c$ is a constant independent of $d$.
Such results have been established in various settings, from the simple case of $n$-independent samples drawn from a given distribution \cite{dudley1969speed,fournier2015rate,weed2019sharp}, to complicated cases of the McKean-Vlasov system~\cite{dos2018simulation} and MGFs \cite{carmona2013probabilistic}. In many interesting applications ({\it e.g.}, the construction of $\varepsilon$-Nash equilibria \cite{carmona2013probabilistic}), $d$ can be so large that the resulting convergence rate is extremely slow. This phenomenon is referred to as the curse of dimensionality (CoD), the main challenge in high-dimensional analysis and algorithms. 

The analysis in \cite{dudley1969speed,fournier2015rate,weed2019sharp,dos2018simulation,carmona2013probabilistic} suggests that the CoD phenomenon is related to the usage of the Wasserstein metric and the type of interaction kernels that drive the interaction between bodies. In fact, it is well-known that the convergence of empirical measures under the Wasserstein metric presents the CoD \cite{dudley1969speed} for any distribution that is absolutely continuous with respect to the Lebesgue measure on $\bR^d$. 

In this paper, we propose a new class of dimension-free metrics for the convergence analysis of mean-field problems. Specifically, we take the form of integral probability metrics (IPMs),
\begin{equation}\label{def_ipm}
    \rmD_\Phi(\mu, \mu') = \sup_{f \in \Phi} \left\vert\int f \ud \mu - f \ud \mu'\right\vert,
\end{equation}
and impose a set of criteria for selecting the test function class $\Phi$ to guarantee that the convergence rate under various settings is dimension-free. Our criteria mainly build on the function class's empirical Rademacher complexity, allowing the test functions to be the reproducing kernel Hilbert spaces (RKHSs), the Barron function space, and flow-induced function spaces, just to name a few. The choice of RKHSs is closely related to the maximum mean discrepancy (MMD)\cite{borgwardt2006integrating} as a tool for statistical tests to check if two sets of observations are generated by the same distribution. Therein, for computational efficiency, the test function space is chosen as the unit ball of RKHSs. 

Beating the CoD is also a central topic in the machine learning community.
One of the core problems in the high-dimensional analysis of machine learning models is identifying an appropriate function space equipped with an appropriate norm that can control the approximation and estimation errors of a particular machine learning model.
This perspective is closely related to the proposed probability metrics, and we should already point out to the readers that all proposed test function classes originate from machine learning models. Reproducing kernel Hilbert spaces are particularly important in statistical learning theory due to the representer theorem established in \cite{kimeldorf1971some}. The
Barron space \cite{weinan2019barron} was introduced for analyzing two-layer neural network models where optimal direct and inverse approximation theorems hold, as well as the {\it a priori} estimate \cite{ma2019priori}. The flow-induced function spaces \cite{weinan2019barron} were introduced for analyzing residual neural networks \cite{he2016deep}, which have wide applications in computer vision and scientific machine learning.

\medskip
Our main results are summarized as follows:
\begin{enumerate}
\item We propose a novel metric to measure the distance of probability measures by imposing selection criteria (Assumption~\ref{function_class_assumption}) for test functions $\Phi$ in \eqref{def_ipm}, yielding a dimension-free metric for the convergence of empirical measures associated with independent samples drawn from a given distribution (Theorem~\ref{thm: iid}); 

\item We generalize the results in \cite{yang2020generalization}:
given a target distribution being a bias potential model and a distribution close to the target measured by our proposed probability metric, we can generate a distribution close to the target in terms of the Wasserstein metric and relative entropy (Theorem~\ref{thm:regularity_improving}).
In this sense, we can transform the empirical measure into a new distribution close to the target in the Wasserstein metric without CoD;

\item The convergence result (Theorem~\ref{thm: iid}) is extended to independent identically distributed (i.i.d.) stochastic processes by imposing assumptions (Assumption~\ref{assump:module}) on their modulus of continuity (Theorem~\ref{thm: process});

\item We give three classes of test functions (reproducing kernel Hilbert spaces, Barron space, and flow-induced function spaces) for which the criteria in Assumption~\ref{function_class_assumption} are satisfied (Theorems~\ref{thm: rkhs i.i.d.}, \ref{thm: Property_Barron_Space} and \ref{thm: compositional});

\item The convergence of the empirical measure associated with an $n$-particle system to the distribution of the McKean-Vlasov stochastic differential equation is shown to be free of CoD (Theorem~\ref{thm:MV-Expectation});

\item We show that the construction of an $\varepsilon$-Nash equilibrium for a homogeneous $n$-player game by its mean-field limit has no CoD (Theorem~\ref{thm:MFG}), {\it i.e.} $\varepsilon$ is independent of $d$.
\end{enumerate}

\medskip
{\bf\noindent Notations.}
We use $\mathcal{P}(\bR^d)$ to denote the space of probability measures on $\RR^d$. $\mathcal{P}^p(\bR^d)$ with $p \geq 1$ denotes the subspace of $\mathcal{P}(\RR^d)$ of probability measures with finite $p^{th}$-moment, {\it i.e.}, $\mu \in \mathcal{P}^p(\bR^d)$ if
\begin{equation}
    M_p(\mu) := \left(\int_{\RR^d} \|x\|^p \ud\mu(x)\right)^{1/p} < +\infty.
\end{equation}
We will primarily work with probability measures with finite first and second moments, i.e., $\mathcal{P}^1(\bR^d)$ and $\mathcal{P}^2(\bR^d)$.
We use $\|\cdot\|$ to denote the Euclidean norm and define the Lipschitz constant with respect to the Euclidean norm:
\begin{equation*}
    \text{Lip}(f) = \sup_{x,y,x\neq y} \frac{|f(x)-f(y)|}{\|x-y\|}.
\end{equation*}
We denote by $\delta_{x_0}$ the delta distribution at $x_0$.

\section{A Class of Integral Probability Metrics}\label{sec:GMMD}

To define a metric on the space of probability measures, one natural approach is to choose a suitable class ${\Phi}$ of test functions on $\bR^d$ and compare the difference of integrals.
\begin{defn}[Integral Probability Metrics {\cite{muller1997integral}}]\label{dfn:GMMD}
Let ${\Phi}$ be a class of measurable functions on $\bR^d$ such that
\begin{equation}
    \sup_{f \in {\Phi}} |f(x)| \le C(1 + \|x\|),
\end{equation}
for a constant $C > 0$ depending on ${\Phi}$.
Then for any $\mu, \mu' \in \mathcal{P}^1(\bR^d)$, the integral probability metric (IPM)  $\rmD_{\Phi}$ associated to $\Phi$ is defined as:
\begin{equation}\label{def:GMMD}
    \rmD_{\Phi}(\mu,\mu') = \sup_{f \in {\Phi}}|\int_{\bR^d}f\rmd(\mu - \mu')|.
\end{equation}
\end{defn}
In \cite{zolotarev1984probability}, these metrics are called \emph{probability metrics with a $\zeta$-structure}. In this paper, following \cite{muller1997integral}, we will stick to the more intuitive terminology IPM. Many probability metrics are based on the comparison of integrals of certain functions, for instance,
\begin{itemize}
    \item the class of 1-Lipschitz functions, which leads to 1-Wasserstein metric $\mathcal{W}_1$;
    \item all functions $\bm{1}_{[t, \infty)}$, $t \in \RR$, which gives the Kolmogorov metric on $\mc{P}(\RR)$; 
    \item all functions $\bm{1}_B$ with $B$ being a Borel set on $\RR^d$, which leads to the total variation metric (in fact the set of continuous functions is sufficient); 
    \item the unit ball of a reproducing kernel Hilbert space (RKHS), which yields the maximum mean discrepancy (MMD) defined in \cite{borgwardt2006integrating}.
\end{itemize}

We are interested in the metrics with ``dimension-free'' properties, for instance, that the empirical measure obtained from $n$ independent samples from a given measure $\mu$ approaches $\mu$ with a convergence speed not depending on $d$ explicitly (for a precise statement, see Theorem~\ref{thm: iid}~(c)). In contrast to working with a particular class of functions (as in $\mc{W}_1$ or MMD), we pose conditions on ${\Phi}$, presented in Assumption~\ref{function_class_assumption}, to fulfill our goal. Later in Section~\ref{sec:testfunction}, we shall discuss several classes of test functions, including RKHS, Barron space, and flow-induced function spaces, where Assumption~\ref{function_class_assumption} is satisfied.

\begin{assu}[function class]\label{function_class_assumption}
The set ${\Phi}$ satisfies the following properties:
\begin{enumerate}
    \item[(a)] If $\mu$ is a signed measure on $\bR^d$, 
    \begin{equation}
        \int_{\bR^d}f \rmd \mu = 0, \; \forall f \in {\Phi} \Rightarrow \mu \equiv 0;
    \end{equation}
    \item[(b)]There exist two constants $A_1 := \sup_{f \in \Phi} \mathrm{Lip}(f) < +\infty$ and $A_2 := \sup_{f \in \Phi} |f(0)| < +\infty$;
    \item[(c)]
   There exists a constant $A_3 > 0$, such that for any $\mathcal{X}= \{x^1,\dots,x^n\} \subset \bR^{d}$, the empirical Rademacher complexity satisfies
    \begin{equation}
        \mathrm{Rad}_n({\Phi},\mathcal{X}) \coloneqq \frac{1}{n}\EE \sup_{f \in {\Phi}} |\sum_{i=1}^n\xi_i f(x^i)| \le \frac{A_3}{n} \sqrt{\sum_{i=1}^n(\|x^i\|^2 + 1)},
    \end{equation}
    where $\xi_1,\dots,\xi_n$ are i.i.d. random variables drawn from the Rademacher distribution, {\it i.e.}, $\mathbb{P}(\xi_i = 1) = \mathbb{P}(\xi_i = -1) = \frac{1}{2}$.
\end{enumerate}
\end{assu}

\begin{rem}
Given any $\Phi$ satisfying Assumption \ref{function_class_assumption}, we can define another function class:
\begin{equation}
    \Phi' = \{f - f(0): f \in \Phi\}.
\end{equation}
Then it is obvious that $\rmD_\Phi(\mu,\mu') = \rmD_{\Phi'}(\mu,\mu')$ for any $\mu,\mu' \in \mathcal{P}^1(\bR^d)$ and $\Phi'$ satisfies that
\begin{enumerate}
\item[(a')] If $\mu$ is a signed measure on $\bR^d$, 
    \begin{equation}
       \mu(\bR^d) = 0, \int_{\bR^d}f \rmd \mu = 0, \; \forall f \in {\Phi'} \Rightarrow \mu \equiv 0;
    \end{equation}
    \item[(b')]$\sup_{f \in \Phi'} \mathrm{Lip}(f) = A_1 < +\infty$ and $ \sup_{f \in \Phi'} |f(0)| = 0$;
    \item[(c')]
  For any $\mathcal{X}= \{x^1,\dots,x^n\} \subset \bR^{nd}$, the empirical Rademacher complexity satisfies
    \begin{equation}
        \mathrm{Rad}_n({\Phi'},\mathcal{X}) \coloneqq \frac{1}{n}\EE \sup_{f \in {\Phi'}} |\sum_{i=1}^n\xi_i f(x^i)| \le \frac{A_3'}{n} \sqrt{\sum_{i=1}^n(\|x^i\|^2 + 1)},
    \end{equation}
    where $\xi_1,\dots,\xi_n$ are i.i.d. random variables drawn from the Rademacher distribution, and $A_3'$ satisfies $A_3' = A_2 + A_3$ with $A_2$, $A_3$ defined from the original class $\Phi$.
\end{enumerate}
    We can replace Assumption \ref{function_class_assumption} by (a'), (b'), and (c'), and all properties of the IPM in this work still hold. We choose to use Assumption \ref{function_class_assumption} which allows $f(0) \neq 0$, and introduce $A_2 = \sup_{f \in \Phi}|f(0)|$, mainly because this is the case for most examples we discussed in Section~\ref{sec:testfunction}. 
\end{rem}

\begin{rem}
Assumption \ref{function_class_assumption} (c) is crucial for overcoming the CoD in our case. It is stronger than the usual estimation of the Rademacher complexity that depends on $n^{-1/2} [\max_{1\le i \le n}\|x^i\|+1]$ rather than $n^{-1/2} \sqrt{\frac{1}{n}\sum_{i=1}^n(\|x^i\|^2+1)}$. Rademacher complexity measures the richness of a function class with respect to a specific probability distribution. It is a powerful tool to bound the generalization error when learning the function in the class through empirical risk minimization. 
It is also closely related to other concepts for measuring the richness of a function class, such as the covering number and fat-shattering dimension, which will be discussed below. In addition, by comparing Theorem~\ref{thm: iid} (b) below and  \cite[Corollaries~3.2 and 3.4]{wojtowytsch2020kolmogorov}), we know that in high dimensions, the richness of any class of functions satisfying Assumption~\ref{function_class_assumption} is relatively small compared to the class of all $A_1$-Lipschitz functions. We refer to~\cite{bartlett2002rademacher,shalev2014understanding} for further information about the Rademacher complexity.
\end{rem}

We recall below the definitions of the covering number and fat-shattering dimension and show how to estimate the Rademacher complexity by them. These estimations provide more criteria for checking if a function class satisfies Assumption \ref{function_class_assumption} (c).
\begin{defn}[covering number]
    Given a function class $\Phi$ on $\bR^d$, $\mathcal{X} = \{x^1,\dots,x^n\} \subset \bR^{d}$ and $\epsilon > 0$, a subset $\hat{\Phi} \subset \Phi$ is a $\epsilon$-covering of $\Phi$ if for any $f \in \Phi$, there exists $\hat{f} \in \hat{\Phi}$ such that 
    \begin{equation}
        \|f - \hat{f}\|_{L^2(\mathcal{X}) }:= \sqrt{\frac{1}{n}\sum_{i=1}^n[f(x^i) - \hat{f}(x^i)]^2}\le \epsilon.
    \end{equation}
    The covering number of $\Phi$ is the cardinality  of the smallest $\epsilon$-covering of $\Phi$:
    \begin{equation}
        \mathcal{C}(\Phi,\epsilon, L^2(\mathcal{X})) = \min\{|\hat{\Phi}|: \hat{\Phi} \text{ is a $\epsilon$-covering of } \Phi\}.
    \end{equation}
\end{defn}
\begin{defn}[fat-shattering dimension]
    Given a function class $\Phi$  on $\bR^d$, $\mathbb{I} \subset \bR^d$ and $\epsilon > 0$. We say $\mathbb{I}$ is $\epsilon$-shattered by $\Phi$ if there exists $h: \mathbb{I} \rightarrow \bR$ such that for any $\mathbb{J} \subset \mathbb{I}$, there exists $f \in \Phi$ such that
    \begin{equation}
        f(x) \le h(x), \forall x \in \mathbb{J},\, f(x) > h(x)+\epsilon, \forall x \in \mathbb{I}\backslash \mathbb{J}.
    \end{equation}
    The fat-shattering dimension of $\Phi$ is the largest cardinality of $\epsilon$-shattering: $$\mathrm{vc}(\Phi,\epsilon) = \max\{|\mathbb{I}|: \mathbb{I} \text{ is $\epsilon$-shattered by } \Phi\}.$$
\end{defn}

If $\Phi$ is a $\{0,1\}$-value function class, the fat-shattering dimension coincides with the classic Vapnik-Chernoveniks (VC) dimension; see \cite{matousek2013lectures} for a detailed introduction of the VC dimension. We refer to \cite{van2014probability} for a detailed introduction of the covering number, VC dimension, and fat-shattering dimension.

The following famous inequality by Dudley \cite{dudley} gives an upper bound of the Rademacher complexity by the covering number:
\begin{equation}
    \mathrm{Rad}_n(\Phi,\mathcal{X}) \le \inf_{\epsilon > 0}\{4\epsilon+ \frac{12}{\sqrt{n}} \int_{\epsilon}^{c}\sqrt{\log\mathcal{C}(\Phi,t,L^2(\mathcal{X}))}\rmd t\},
\end{equation}
where $c = \sup_{f \in \Phi} \sqrt{\sum_{i=1}^n f^2(x_i)/n}$. 
Let $p > 0$. We will use the $O(\cdot)$ notation to ignore the constant term which depends only on $p$. Now assume that the function class $\Phi$ satisfies 
\begin{equation}
    \log \mathcal{C}(\Phi, t ,L^2(\mathcal{X})) \le O(\frac{c^p}{t^p}),
\end{equation}
then one has that when $p > 2$,
\begin{equation}
      \mathrm{Rad}_n(\Phi,\mathcal{X}) \le O(\inf_{\epsilon > 0}\{\epsilon + c^{\frac{p}{2}}n^{-\frac{1}{2}}\epsilon^{1-\frac{p}{2}}- cn^{-\frac{1}{2}}\}) \le O(cn^{-\frac{1}{p}});
\end{equation}
when $p = 2$,
\begin{equation}
    \mathrm{Rad}_n(\Phi,\mathcal{X}) \le O(\inf_{\epsilon > 0}\{\epsilon + cn^{-\frac{1}{2}}[\log c - \log \epsilon]\})\le O(cn^{-\frac{1}{2}} + cn^{-\frac{1}{2}}[\log c - \log c n^{-\frac{1}{2}}]) = O(cn^{-\frac{1}{2}}\log n);
\end{equation}
and when $0 < p < 2$,
\begin{equation}
    \mathrm{Rad}_n(\Phi,\mathcal{X}) \le O(\inf_{\epsilon > 0}\{\epsilon + c n^{-\frac{1}{2}} - c^{\frac{p}{2}} n^{-\frac{1}{2}}\epsilon^{1-\frac{p}{2}}\}) = O(c n^{-\frac{1}{2}}).
\end{equation}
In summary,
\begin{equation}\label{covering}
    \mathrm{Rad}_n(\Phi,\mathcal{X}) \le \begin{dcases}
        O(c n^{-\frac{1}{p}}), &\text{ when } p > 2\\
        O(c n^{-\frac{1}{2}}\log n), &\text{  when } p = 2\\
        O(c n^{-\frac{1}{2}}), &\text{ when } 0 < p < 2.
    \end{dcases}
\end{equation}
Noticing that if $\Phi$ satisfies Assumption \ref{function_class_assumption} (b), we have that
\begin{equation}
    c = \sup_{f \in \Phi} \sqrt{\sum_{i=1}^n \frac{f^2(x_i)}{n}}\le \frac{\sqrt{2}\max\{A_1,A_2\}}{\sqrt{n}}\sqrt{\sum_{i=1}^n[\|x^i\|^2+1]}.
\end{equation}
Therefore, when $ p < 2$, the function class $\Phi$ satisfies Assumption \ref{function_class_assumption} (c). When $p \ge 2$, the convergence rate of the Rademacher complexity in \eqref{covering} is slower than the rate $n^{-\frac12}$ in Assumption \ref{function_class_assumption} (c), but most results in this work still hold if we change the convergence rate to $n^{-\frac{1}{p}}$ or $n^{-\frac{1}{2}}\log n$ in these results. 

The fat-shattering dimension can be used to bound the covering number and hence bound the Rademacher complexity. Assume that there exists $\xi: \mathbb{R}^+ \rightarrow \mathbb{R}^+$ and $\alpha > 1$ such that
\begin{equation}
    \mathrm{vc}(\Phi,\epsilon)\le \xi(\epsilon), \quad \xi(\alpha\epsilon)\le \frac{\xi(\epsilon)}{8},
\end{equation}
then for any $\mathcal{X} = \{x^1,\dots,x^n\} \subset \bR^d $, there exists $C > 0$ depending only on $\alpha$ such that
\begin{equation}
    \log\mathcal{C}(\Phi,\epsilon, L^2(\mathcal{X})) \le C \xi(C\epsilon).
\end{equation}
For further details, see \cite[Theorem 1.3]{rudelson2006combinatorics} and \cite[Corollary 7.48]{van2014probability}.

In our case, if $\xi(\epsilon) = \epsilon^{-p}$, one can obtain that  $\log\mathcal{C}(\Phi,\epsilon, L^2(\mathcal{X})) = O(\epsilon^{-p})$ and then use inequality \eqref{covering} to bound the Rademacher complexity.

The subsection below highlights properties of the metric $\rmD_{\Phi}$ for random variables when the function class $ {\Phi}$ satisfies Assumption~\ref{function_class_assumption}.
For comparison purpose, we recall the $p$-Wasserstein metric $\mathcal{W}_p$ defined as follows
\begin{equation}\label{def:wp}
    \mathcal{W}_p(\mu,\mu') = \Bigl(\inf_{\gamma \in \Gamma(\mu,\mu')}\int_{\bR^d\times\bR^d}\|x-y\|^p\rmd \gamma(x,y)\Bigr)^{1/p},
\end{equation}
where $\mu,\mu' \in \mathcal{P}^p(\bR^d)$ and $\Gamma(\mu,\mu')$ denotes the collection of all probability distributions on $\bR^d \times \bR^d$ with marginals $\mu$ and $\mu'$ on the first and second arguments, respectively. Note that Definition~\ref{dfn:GMMD} with $\Phi = \{\text{all continuous 1-Lipschitz functions from }\RR^d \text{ to } \RR\}$ admits the dual representation of \eqref{def:wp} with $p = 1$ ({\it cf.} \cite[Theorem 1.3]{basso2015hitchhiker}). We also introduce the relative entropy or Kullback-Leibler divergence \begin{equation}
    \mathcal{H}(\mu | \mu') \coloneqq \EE^\mu\left(\log(\frac{\ud  \mu}{\ud \mu'})\right),
\end{equation}
and the total variation distance
\begin{equation}
    \|\mu - \mu'\| = \sup_{A \subset \bR^d: \text{Borel measurable set}}|\mu(A) - \mu'(A)|.
\end{equation}

\subsection{Convergence Analysis for Random Variables}
\begin{thm}\label{thm: iid}
 Under Assumption \ref{function_class_assumption}, we have:

\begin{enumerate}
    \item[(a)] $\rmD_{\Phi}$ is a metric on $\mathcal{P}^1(\bR^d)$. In addition, if $\mu,\mu' \in \mathcal{P}^2(\bR^d)$,
    \begin{equation}
        \rmD_{\Phi}(\mu, \mu') \le A_1 \mathcal{W}_1(\mu, \mu') \le A_1 \mathcal{W}_2(\mu, \mu').
    \end{equation}
    
\item[(b)] Let $K$ be a compact set in $\bR^d$, $\{\mu_i\}_{i=1}^{\infty}, \mu \in \mathcal{P}(\bR^d)$ such that $\mu_i(K) = \mu(K) = 1$ for $ i \ge 1$. Then the following statements are equivalent:
\begin{itemize}
    \item[1.] $\mu_i$ converges to $\mu$ in the weak sense;
    \item[2.] $\lim_{i\rightarrow \infty}\rmD_\Phi(\mu_i,\mu) = 0$.
\end{itemize}
In other words, $\rmD_\Phi$ metrizes the weak convergence of measures on compact sets.

    \item[(c)]   
    Given $\mu \in \mathcal{P}^2(\bR^d)$, let $X^1,\dots,X^n$ be i.i.d. random variables drawn from the distribution $\mu$ and 
\begin{equation}
    \bar{\mu}^n = \frac{1}{n}\sum_{i=1}^n \delta_{X^i}
\end{equation}
be the empirical measure of $X^1, \ldots, X^n$. Then,
\begin{align}
    &\EE \rmD_{{\Phi}}(\mu,\bar{\mu}^n)
    \le \frac{2A_3}{\sqrt{n}} \sqrt{\int_{\bR^d}[\|x\|^2+1] \rmd \mu(x)},\\
    &\EE[\rmD_{\Phi}^2(\mu, \bar{\mu}^n)]\le \frac{A_1^2}{n} \int_{\bR^d}\|x\|^2 \rmd \mu(x) + \frac{4A_3^2}{n}\int_{\bR^d}[\|x\|^2+1] \rmd \mu(x).
\end{align}

\item[(d)] If $\mu$ satisfies the $\mathrm{T}_1$ inequality ({\it cf.} \cite{gozlan2010transport}), 
\begin{equation}
    \mathcal{W}_1^2(\mu,\tilde{\mu}) \le 2\kappa^2\mathcal{H}(\tilde{\mu}|\mu) \quad \forall \tilde{\mu} \ll \mu, 
\end{equation}
(an equivalent condition of the $\mathrm{T}_1$ inequality is that, there exists a constant $\delta > 0$ such that $\int_{\bR^d}\int_{\bR^d}\exp(\delta \|x-y\|^2)\rmd \mu(x) \rmd \mu(y)< +\infty$; see \cite[Theorem 2.3]{djellout2004transportation}),
then,
\begin{equation}
    \mathbb{P}\left(\rmD_{\Phi}(\mu,\bar{\mu}^n) -\frac{2A_3}{\sqrt{n}} \sqrt{\int_{\bR^d}[\|x\|^2+1] \rmd \mu(x)} \ge a\right) \le \exp\left(-\frac{na^2}{2A_1^2\kappa^2}\right).
\end{equation} 
\end{enumerate}
\end{thm}
\begin{proof}
For claim (a), by definition $\rmD_{\Phi}$ is symmetric and satisfies the triangle inequality. We only need to show that for any $\mu, \mu' \in \mathcal{P}^1(\bR^d)$,  $\rmD_{\Phi}(\mu, \mu') = 0$ leads to $\mu = \mu'$, which is ensured directly by Assumption~\ref{function_class_assumption}~(a). In addition, from Assumption~\ref{function_class_assumption}~(b), we deduce that $\rmD_{\Phi}(\mu_1,\mu_2)\le A_1 \mathcal{W}_1(\mu_1,\mu_2)$ by using the definition of $\mathcal{W}_1$. 
The relation between $\mathcal{W}_1$ and $\mathcal{W}_2$ is a classical result following from Jensen's inequality.

For claim (b), we first prove $1 \Rightarrow 2$. By Theorem 8.3.2 in \cite{bogachev2007measure}, we know that
\begin{equation}
    \lim_{i\rightarrow \infty} \mathcal{W}_1(\mu_i, \mu)  = 0.
\end{equation}
Then noticing that $\rmD_\Phi(\mu_i,\mu) \le A_1 \mathcal{W}_1(\mu_i,\mu)$, we obtain the result. For the claim $2 \Rightarrow 1$, we first notice that $\{\mu_i\}_{i=1}^{\infty}$ is tight. Then by Theorem 8.6.2 in \cite{bogachev2007measure}, every subsequence of $\{\mu_i\}_{i =1}^{\infty}$ has a weakly convergent subsequence, denoted by $\{\mu_{i_k}\}_{k=1}^{\infty}$. Let $\mu'$ be the limit of $\{\mu_{i_k}\}_{k=1}^{\infty}$. Then for any $f\in \Phi$,
\begin{equation}
    \lim_{k \rightarrow \infty} \int_{\bR^d} f \rmd \mu_{i_k} = \int_{\bR^d} f \rmd \mu'.
\end{equation}
Meanwhile, since $\lim_{k\rightarrow \infty}\rmD_\Phi(\mu_{i_k},\mu) = 0$, we have 
\begin{equation}
    \lim_{k \rightarrow \infty} \int_{\bR^d} f \rmd \mu_{i_k} = \int_{\bR^d} f \rmd \mu.
\end{equation}
Hence, for any $f \in \Phi$, $\int_{\bR^d} f \rmd \mu = \int_{\bR^d} f \rmd \mu'$, which means that $\mu \equiv \mu'$. In other words, every subsequence of $\{\mu_i\}_{i =1}^{\infty}$ has a weakly convergent subsequence to $\mu$. Since the weakly convergence can be metrized \cite[Theorem~8.3.2]{bogachev2007measure}, we know that $\{\mu_i\}_{i=1}^{\infty}$ weakly converge to $\mu$.

For claim (c), we first notice that
\begin{equation}
    \EE[\rmD_{\Phi}(\mu,\bar{\mu}^n)] = \frac{1}{n}\EE\left[\sup_{f \in {\Phi}}\Bigl|\sum_{i=1}^n [f(X^i) - \EE f(X^i)]\Bigr|\right]\le   \frac{2}{n}\EE\left[\sup_{f \in {\Phi}} \Bigl|\sum_{i=1}^n\xi_i f(X^i)\Bigr|\right],
\end{equation}
where in the last step we use the Rademacher complexity to bound the largest gap between the expectation of a function and its empirical version (see, {\it e.g.},~\cite[ Lemma 26.2]{shalev2014understanding}).
Using Assumption~\ref{function_class_assumption} (c),
one deduces
\begin{align}\label{Rademacher Expectation}
    \EE \rmD_{{\Phi}}(\mu,\bar{\mu}^n)&
    \le \frac{2A_3}{n}\EE\sqrt{\sum_{i=1}^n(\|X^i\|^2+1)} \le  \frac{2A_3}{n}\sqrt{\EE \left[\sum_{i=1}^n(\|X^i\|^2+1)\right]} 
    = \frac{2A_3}{\sqrt{n}} \sqrt{\int_{\bR^d}[\|x\|^2+1] \rmd \mu(x)}.
\end{align}
Now let $Y^1,\dots,Y^n$ be i.i.d. random variables drawn from the distribution $\mu$ and be independent of $X^1,\dots,X^n$.
By the Efron-Stein-Steele inequality ({\it cf.}~\cite[Theorem~5]{boucheron2003concentration}) 
that reads
\begin{equation}
    \mathrm{Var} (Z) \leq \frac{1}{2}\sum_{i=1}^n \EE[Z - Z_i']^2,
\end{equation}
for some measurable function $g$ of $n$ variables, $Z = g(X^1, \ldots, X^n)$ and $Z'_i = g(X^1, \ldots, Y^i, \ldots, X^n)$, and the uniform Lipschitz property of $f \in \Phi$, one has
{\small
\begin{align}
    &\mathrm{Var}(\rmD_{{\Phi}}(\mu,\bar{\mu}^n)) \\
    =~& \mathrm{Var}\left(\frac{1}{n}\sup_{f \in {\Phi}}\Bigl|\sum_{i=1}^n [f(X^i) - \EE f(X^i)]\Bigr|\right) \notag \\
    \le~& \frac{1}{2}\sum_{i=1}^n\EE\left[\left(\frac{1}{n}\sup_{f \in {\Phi}}\Bigl|\sum_{j=1}^n [f(X^j) - \EE f(X^j)]\Bigr| - \frac{1}{n}\sup_{f \in {\Phi}}\Bigl|\sum_{j=1,j\neq i}^n [f(X^j) - \EE f(X^j)] + f(Y^i) - \EE f(X^i)\Bigr|\right)^2\right] \notag \\
    \le~& \frac{A_1^2}{2n^2}\sum_{i=1}^n \EE\|X^i - Y^i\|^2 = \frac{A_1^2}{n}\int_{\bR^d}\|x\|^2 \rmd \mu(x).
\end{align}}
Therefore, we have
\begin{equation}
    \EE[\rmD^2_{{\Phi}}(\mu,\bar{\mu}^n)] = \mathrm{Var}(\rmD_{{\Phi}}(\mu,\bar{\mu}^n)) +    [\EE \rmD_{{\Phi}}(\mu,\bar{\mu}^n)]^2 \le \frac{A_1^2}{n}\int_{\bR^d}\|x\|^2 \rmd \mu(x) + \frac{4A_3^2}{n}\int_{\bR^d}[\|x\|^2+1] \rmd \mu(x). 
\end{equation}

For claim (d), we need some established concentration inequalities.
 Let $\mu^{\otimes n} \in \mathcal{P}(\RR^{d\times n})$ be the $n^{th}$ times product of $\mu$ and the distance between $\bx = (x^1,\dots,x^n),\bm{y}=(y^1,\dots,y^n) \in \bR^{d\times n}$ be 
\begin{equation}
    \|\bx - \bm{y}\| = \sum_{i=1}^{n}\|x^i - y^i\|.
\end{equation}
By Theorem 5.2 in \cite{delarue2020master}, for any $\tilde{\mu} \ll \mu^{\otimes n}$, we have
\begin{equation}
    \mathcal{W}_1^2(\mu^{\otimes n},\tilde{\mu}) \le 2n\kappa^2 \mathcal{H}(\tilde{\mu}|\mu^{\otimes n}).
\end{equation}
Define
\begin{equation}
    G(\bx) = \sup_{f \in \Phi}|\frac{1}{n}\sum_{i=1}^nf(x^i) - \int_{\bR^d}f\rmd\mu|.
\end{equation}
Combining the fact that
\begin{equation}
    |G(\bx) - G(\bm{y})| \le \frac{A_1}{n}\|\bx - \bm{y}\|,
\end{equation}
for any $\bx $ and $\bm{y}  \in \bR^{d \times n}$ with Theorem 5.1 in \cite{delarue2020master}, we have
\begin{equation}
    \PP\left(\rmD_{\Phi}(\mu,\bar{\mu}^n) - \EE \rmD_{\Phi}(\mu,\bar{\mu}^n)\ge a\right) = \mu^{\otimes n}\Bigl(G(\bx) - \int_{\bR^{d\times n}}G(\bx)\rmd \mu^{\otimes n}(\bx) \ge a\Bigr) \le \exp\left(-\frac{na^2}{2A_1^2\kappa^2}\right).
\end{equation}
Finally, we conclude our proof by the observation
\begin{equation}
    \EE \rmD_{\Phi}(\mu,\bar{\mu}^n) \le \frac{2A_3}{\sqrt{n}} \sqrt{\int_{\bR^d}[\|x\|^2+1] \rmd \mu(x)},
\end{equation}
from inequality \eqref{Rademacher Expectation}.
\end{proof}

The above theorem focuses on the convergence of $\bar \mu^n$ to $\mu$ in the sense of IPM satisfying Assumption~\ref{function_class_assumption}. It is well-known that the convergence of empirical measures under the Wasserstein metric faces the CoD \cite{dudley1969speed} for any distribution that is absolutely continuous with respect to the Lebesgue measure on $\bR^d$. However, as shown in \cite{yang2020generalization,yang2022generalization}, if the target distribution admits a specific representation, the bias potential model or the density model, we can use the empirical measure and the representation form to generate a new distribution that is close to the target distribution in the sense of the Wasserstein metric, total variation distance, or relative entropy. The following theorem, generalizing results in \cite{yang2020generalization,yang2022generalization}, shows that given a bias potential model or a density model as the target distribution and a distribution close to the target measured by our proposed IPM, we can generate a distribution close to the target one in terms of the Wasserstein metric, total variation distance or relative entropy.

\begin{thm}\label{thm:regularity_improving}
Suppose $P \in \mathcal{P}(\bR^d)$ is a known base distribution and $\Phi$ is a function class satisfying Assumption \ref{function_class_assumption}. 
\begin{enumerate}
    \item[(a)] Assume the target distribution $\mu \in \MCP^1(\bR^d)$ satisfies the bias potential model - that is, there exists $V^* \in \Phi$ such that
   \begin{equation}
        \int_{\bR^d} e^{-V^*(x)}\rmd P(x) < +\infty,\, \frac{\rmd \mu}{\rmd P} = \frac{e^{-V^*}}{\int_{\bR^d}e^{-V^*(x)}\rmd P(x)}.
    \end{equation}
     Let $\nu \in \mathcal{P}^1(\bR^d)$ be an accessible probability measure, define
     \begin{equation}
            \mathrm{Loss}_1(V) = \int_{\bR^d} V(x)\rmd \nu(x) + \log\left(\int_{\bR^d} e^{-V(x)}\rmd P(x)\right), 
    \end{equation}
    for $V \in \Phi$. Let $V' \in \Phi$ and $\mu' \in \mathcal{P}(\bR^d)$ satisfy
    \begin{equation}
         \int_{\bR^d} e^{-V'(x)}\rmd P(x) < +\infty,\, \frac{\rmd \mu'}{\rmd P} = \frac{e^{-V'}}{\int_{\bR^d}e^{-V‘’}\rmd P(x)}.
    \end{equation}
    Then $\inf_{V\in \Phi} \mathrm{Loss}_1(V) > -\infty$ and
    \begin{equation}
        \mathcal{H}(\mu|\mu') \le 2 \rmD_\Phi(\mu,\nu) + \mathrm{Loss}_1(V') - \inf_{V\in\Phi}\mathrm{Loss}_1(V).
    \end{equation}
     Assume additionally that $P$ has a compact support
     \begin{equation}
         K = \max\{\|x-y\|, x,y \in \mathrm{supp}(P)\} < +\infty,
     \end{equation}
     then
     \begin{equation}
         \mathcal{W}_1^2(\mu,\mu')\le \frac{K^2}{2}[2 \rmD_\Phi(\mu,\nu) + \mathrm{Loss}_1(V') - \inf_{V\in\Phi}\mathrm{Loss}_1(V)].
     \end{equation}
    \item[(b)] Assume the target distribution $\mu \in \MCP^1(\bR^d)$ satisfies the density model - that is, there exists $q^* \in \Phi$ such that
  $
       q^*= \frac{\rmd \mu}{\rmd P}.
   $
     Let $\nu \in \mathcal{P}^1(\bR^d)$ be an accessible probability measure, define
     \begin{equation}
            \mathrm{Loss}_2(q) = \sup_{f \in \mathcal{F}}|\int_{\bR^d}f\rmd \nu - \int_{\bR^d}fq\rmd P|, 
    \end{equation}
    for $q \in \Phi$. Let $q' \in \Phi$ and $\mu' \in \MCP(\bR^d)$ satisfy $q' = \frac{\rmd \mu'}{\rmd P}$. Then
    \begin{equation}
        \inf_{q \in \Phi}\mathrm{Loss}_2(q) \le \rmD_\Phi(\mu,\nu),
    \end{equation}
    and
    \begin{equation}
        \|\mu-\mu'\|_{TV}^2\le 2[\rmD_\Phi(\mu,\nu) + \mathrm{Loss}_2(q')].
    \end{equation}
     Assume additionally that $P$ has a compact support
     \begin{equation}
         K = \max\{\|x-y\|, x,y \in \mathrm{supp}(P)\} < +\infty,
     \end{equation}
     then
     \begin{equation}
         \mathcal{W}_1^2(\mu,\mu')\le 2K^2[\rmD_\Phi(\mu,\nu) + \mathrm{Loss}_2(q')].
     \end{equation}
\end{enumerate}
\end{thm}
\begin{proof}
For the claim (a), first noticing that  $\mathrm{Loss}_1(V)$ is lower bounded for $V\in \Phi$ using $|V(x)| \le A_1\|x\| + A_2$, one know that
\begin{align}
    \inf_{V \in \Phi} \mathrm{Loss}_1(V) &\ge -\int_{\bR^d}[A_1\|x\| + A_2]\rmd \nu(x) + \log(\int_{\bR^d}e^{-[A_1\|x\|+A_2]}\rmd P(x)) \\
    &\ge -\int_{\bR^d}[A_1\|x\| + A_2]\rmd \nu(x) + \log(\int_{\|x\| \le M}e^{-[A_1\|x\|+A_2]}\rmd P(x)) \\
    &\ge -\int_{\bR^d}[A_1\|x\| + A_2]\rmd \nu(x) + \log(e^{-MA_1 - A_2} P(\|x\| \le M)) \\
    &= -\int_{\bR^d}[A_1\|x\| + A_2]\rmd \nu(x) -(MA_1+A_2) +  \log(P(\|x\| \le M)) > -\infty,
\end{align}
where $M$ is large enough such that $P(\|x\|\le M) > 0$.
It is then easy to compute 
\begin{align}
    \mathcal{H}(\mu|\mu') &= \int_{\bR^d}(V'-V^*)(x)\rmd \mu(x) - \log  \left(\int_{\bR^d} e^{-V^*(x)}\rmd P(x)\right) + \log \left(\int_{\bR^d}e^{-V'(x)}\rmd P(x)\right) \notag
    \\ &\le 2 \rmD_\Phi(\mu,\nu) + \mathrm{Loss}(V') - \mathrm{Loss}(V^{*}) \le  2\rmD_\Phi(\mu,\nu)  + \mathrm{Loss}(V') - \inf_{V \in \Phi} \mathrm{Loss}(V) .
\end{align}
To estimate the Wasserstein metric between $\mu$ and $\mu'$, first notice the Pinsker's inequality \cite{pinsker1964information,csiszar1967information}:
\begin{equation}
    \|\mu-\mu'\|_{TV}^2 \le \frac{1}{2}\mathcal{H}(\mu|\mu').
\end{equation}
Recalling the definition of the 1-Wasserstein metric \eqref{def:wp} and $\mathrm{supp}(\mu),\mathrm{supp}(\mu') \in \mathrm{supp}(P)$, we have
\begin{equation}\label{thm2.2: eq1}
    \mathcal{W}_1(\mu,\mu')  = \inf_{\gamma \in \Gamma(\mu,\mu')}\int_{\bR^d\times\bR^d}\|x-y\|\rmd \gamma(x,y) \le K\inf_{\gamma \in \Gamma(\mu,\mu')}\int_{\bR^d\times\bR^d}\mathrm{1}_{x\neq y}\rmd \gamma(x,y) = K\|\mu-\mu'\|_{TV},
\end{equation}
which concludes the proof.
For the claim (b), first noticing that 
\begin{equation}
    \rmD_\Phi(\mu,\nu) = \mathrm{Loss}_2(q^*) \geq \inf_{q \in \Phi}\mathrm{Loss}_2(q).
\end{equation}
By the triangular inequality,
\begin{equation}
    \rmD_{\Phi}(\mu,\mu') \le \rmD_{\Phi}(\mu,\nu) + \rmD_{\Phi}(\nu,\mu') = \rmD_{\Phi}(\mu,\nu) + \mathrm{Loss}_2(q').
\end{equation}
Noticing that
\begin{equation}
    \rmD_{\Phi}(\mu,\mu') = \sup_{f \in \Phi}|\int_{\bR^d}f(q^* - q')\rmd P| \ge \frac{1}{2}\int_{\bR^d}|q^* - q'|^2\rmd P \ge \frac{1}{2}[\int_{\bR^d}|q^*-q|\rmd P]^2 = \frac{1}{2}\|\mu-\mu'\|_{TV}^2,
\end{equation}
as $\frac{1}{2}(q^* - q') \in \Phi$. The rest of claim (b) follows \eqref{thm2.2: eq1}.
\end{proof}

Theorem \ref{thm:regularity_improving} shows that if one knows that the target distribution satisfies a bias potential model or density model and can access to a distribution which is close to the target distribution with respect to a IPM (usually the empirical distribution by Theorem \ref{thm: iid}), one can then solve the optimization problem 
\begin{equation}
    \inf_{V \in \Phi} \mathrm{Loss}_1(V)\quad
\text{or} \quad
    \inf_{q\in \Phi} \mathrm{Loss}_2(q),
\end{equation}
and the resulting distribution is close to the target distribution with respect to the Wasserstein metric, total variation distance or relative entropy.
Besides the bias potential model and density model we study here, another important form of distribution representation is the generative adversarial network (GAN) \cite{goodfellow2014generative,arjovsky2017wasserstein}, which assumes $\mu = \phi\circ P$ (this notation stands for the measure $P$ push-forwarded by $\phi$) with a known base distribution $P \in \mathcal{P}(\bR^d)$ and $\phi$ lies in certain function classes. In practice, GAN has shown astonishing power in learning distribution \cite{brock2018large,donahue2019large}. However, how to establish similar theoretical results with respect to GAN is far from clear.

\subsection{Convergence Analysis for Stochastic Processes}

Theorems in the previous subsection focus on the measure $\mu$ on $\RR^d$. When dealing with empirical measures of a stochastic process such as discussing the convergence of the $n$-particle dynamics to the McKean-Vlasov system (see Section~\ref{sec:mvsde} for details), {\it i.e.},  $\mu \in \mathcal{P}(C([0,T]; \RR^d))$, we need to extend Theorem \ref{thm: iid} to Theorem \ref{thm: process} below. The results are similar, except for an additional term $\phi(n)$ coming from the regularity of the process. To this end, we introduce the following assumption.

\begin{assu}[modulus of continuity]\label{assump:module}
Given a probability measure $\mu$ on $C([0,T];\bR^d)$, there exist constants  $Q,\alpha,\beta > 0$, such that for any $h > 0$,
\begin{equation}\label{eq:module}
    \int_{C([0,T];\bR^d)}\Delta(x,h)\rmd \mu(x) \le Q h^\alpha \log^\beta\left(\frac{2T}{h}\right),
\end{equation}
where $\Delta(f, h)$ denotes the modulus of continuity of $f \in C([0,T]; \RR^d)$:
\begin{equation}\label{module of continuity}
    \Delta(f,h) = \sup_{t,s \in [0,T],|t-s|\le h}\|f(t) - f(s)\|.
\end{equation}
\end{assu}

\begin{rem}\label{rem_process_continuity}
 The condition \eqref{eq:module} is satisfied by many kinds of processes. For instance, see \cite[Theorem~1]{fischer2009moments} for results for Brownian motion and It\^{o} diffusion processes with $\alpha = \beta = \frac{1}{2}$, under quite generic conditions on the drift and diffusion coefficients. Note that the constant $Q$ derived therein may depend on the dimension $d$, while the result in our Theorem~\ref{thm: prior_estimation} does not.
\end{rem}

\begin{thm}\label{thm: process}
Under Assumptions~\ref{function_class_assumption} and \ref{assump:module}, 
\begin{enumerate}
    \item[(a)] Let $\mu$ be a probability distribution on $C([0,T];\bR^d)$ such that
\begin{equation}
    \int_{C([0,T];\bR^d)}\sup_{0\le t \le T}\|x_t\|^2\rmd \mu(x) < +\infty.
\end{equation}
Denote by $X^1, \cdots, X^n$ i.i.d. random processes drawn from $\mu$, and
 \begin{equation}
        \bar{\mu}_t^n = \frac{1}{n}\sum_{i=1}^n \delta_{X_t^i}
\end{equation}
as the empirical measure of $X_t^1,\dots,X_t^n$. Define $\mu_t = \MCL(X_t^1)$, then 
\begin{equation}\label{eq:EDprocess}
    \EE \left[\sup_{0\le t \le T}\rmD_{\Phi}(\mu_t,\bar{\mu}_t^n)\right] \le \phi(n),
\end{equation}
where
{\small\begin{align}
    \phi(n) =~& 2\left[\frac{A_3 + 8\max\{A_1,A_2\}}{\sqrt{n}} + 8\max\{A_1,A_2\}\sqrt{\frac{\log(2^{2\alpha}n)}{2\alpha n}}\right]\left(\int_{C([0,T];\bR^d)}\sup_{0\le t \le T}\|x_t\|^2\rmd \mu(x) +1\right)^{\frac{1}{2}}\notag \\
    &+ 2\frac{A_1Q T^\alpha}{\sqrt{n}}\left(\frac{\log(2^{4\alpha}n)}{2\alpha}\right)^{\beta} \notag \\
    =~& O(n^{-\frac{1}{2}}(\log n)^{\max\{\beta,\frac{1}{2}\}}).
\end{align}}
In addition,
\begin{equation}
        \EE \left[\sup_{0\le t \le T}\rmD^2_{\Phi}(\mu_t,\bar{\mu}_t^n)\right] \le \frac{A_1^2}{n}\int_{C([0,T];\bR^d)}\sup_{0\le t \le T}\|x_t\|^2\rmd \mu(x) + \phi^2(n).
    \end{equation}

\item[(b)] If $\mu$ satisfies the $\mathrm{T}_1$ inequality, that is,
\begin{equation}
    \mathcal{W}_1^2(\mu,\tilde{\mu}) \le 2\kappa^2\mathcal{H}(\tilde{\mu}|\mu) \quad \forall \tilde{\mu} \ll \mu, 
\end{equation}
where $\mathcal{H}$ is the relative entropy defined in Theorem~\ref{thm: iid} (d), then
\begin{equation}
    \mathbb{P}\left(\sup_{0\le t\le T}\rmD_{\Phi}(\mu_t,\bar{\mu}_t^n) - \phi(n) \ge a\right) \le \exp\left(-\frac{na^2}{2A_1^2\kappa^2}\right).
\end{equation} 
\end{enumerate}
\end{thm}
 
 \begin{rem}\label{rem:log}
The logarithmatic term in $\phi(n)$ can be removed for many kinds of function classes $\Phi$, in particular for the function classes discussed in Section~\ref{sec:testfunction}. We defer the proof of this claim to \ref{app:functionclass}. 
 \end{rem}
\begin{proof}
The proof can be reduced to establishing \eqref{eq:EDprocess}. The estimate of $\EE \left[\sup_{0\le t \le T}\rmD_{\Phi}^2(\mu_t,\bar{\mu}_t^n)\right]$ and (b) can be indeed derived from \eqref{eq:EDprocess}, and by following the proof arguments of Theorem~\ref{thm: iid}. The class $\Phi$ in Assumption~\ref{function_class_assumption} is invariant up to an additive constant and so without loss of generality, $\EE \left[\sup_{0\le t \le T}\rmD_{\Phi}(\mu_t,\bar{\mu}_t^n)\right]$ can be reduced to
 \begin{equation}
     \frac{1}{n}\EE\sup_{t \in [0,T]} \sup_{f \in {\Phi}}\Bigl| \sum_{i=1}^n\xi_i f(X_t^i)\Bigr|,
 \end{equation}
where $\xi_1,\dots,\xi_n$ are i.i.d. random variables drawn from the Rademacher distribution and are independent of the processes $X^1,\dots,X^n$.

To this end,  given any $y^1,\dots,y^n \in \bR^d$, we define $F(\xi)$ for $\xi = (\xi_1,\dots,\xi_n) \in \{-1,1\}^{\otimes n}$ by
 \begin{equation}
     F(\xi) \coloneqq \frac{1}{n}\sup_{f \in {\Phi}}\Bigl|\sum_{i=1}^n\xi_i f(y^i)\Bigr|.
 \end{equation}
  By Assumption~\ref{function_class_assumption}~(c), we immediately have
 \begin{equation}\label{F_expectation_estimation}
     \EE F(\xi) \le \frac{A_3}{n}\sqrt{\sum_{i=1}^n(\|y^i\|^2 + 1)}.
 \end{equation}
 By definition, for any $\xi$ and $\epsilon > 0$, there exists a function $f^\xi \in {\Phi}$ such that
 \begin{equation}
     F(\xi) \le \frac{1}{n}\Bigl|\sum_{i=1}^n\xi_i f^\xi(y^i)\Bigr| + \epsilon.
 \end{equation}
 Note that, for any $1 \le j \le n$, one has
 \begin{align}
   D_jF(\xi) &\coloneqq   F(\xi) - \min_{z\in\{0,1\}}F(\xi_1,\dots,\xi_{j-1},z,\xi_{j+1},\dots,\xi_n)\notag\\
   &\le \epsilon + \frac{1}{n}\Bigl|\sum_{i=1}^n\xi_i f^\xi(y^i)\Bigr|- \min_{z\in\{0,1\}}F(\xi_1,\dots,\xi_{j-1},z,\xi_{j+1},\dots,\xi_n)\notag\\
   &\le \epsilon + \frac{1}{n}\Bigl|\sum_{i=1}^n\xi_i f^\xi(y^i)\Bigr|- \min_{z\in\{0,1\}}\frac{1}{n}\Bigl|\sum_{i \neq j}\xi_i f^\xi(y^i) +z f^\xi(y^j)\Bigr| \notag \\
   &\le \epsilon + \frac{2}{n}|f^\xi(y^j)|.
 \end{align}
 Taking the square and summing the above inequality over $j$ and next taking the supremum over all $\xi$ gives, 
 \begin{equation}
     \sup_{\xi} \sum_{j=1}^n |D_j F(\xi)|^2 \le \sup_{\xi} \sum_{j=1}^n\left(\epsilon + \frac{2}{n}|f^\xi(y^j)|\right)^2 \le \sup_{f \in {\Phi}} \sum_{j=1}^n \left(\epsilon + \frac{2}{n}|f(y^j)|\right)^2,
 \end{equation}
 holding true for arbitrary $\eps$, which implies 
 \begin{align}
     \sup_{\xi} \sum_{j=1}^n |D_j F(\xi)|^2 \le \frac{4}{n^2}\sup_{f \in {\Phi}}\sum_{j=1}^n|f(y^j)|^2 \le \frac{8}{n^2}\sum_{j=1}^n (A_1^2\|y^j\|^2 + A_2^2).
 \end{align}
Therefore, by the bounded difference inequality (\cite[Theorem 3.18]{van2014probability}), we obtain
 \begin{equation}
 \label{F_tail_prob}
    \mathbb{P}(F(\xi) - \EE F(\xi) \ge a) \le \exp\left(-\frac{n^2a^2}{64\sum_{j=1}^n(A_1^2\|y^j\|^2 + A_2^2)}\right).
 \end{equation}
Given $n' \in \mathbb{N}^+$, $x^1, \dots, x^{n}\in C([0,T];\bR^d)$ and $t_1,\dots,t_{n'} \in [0,T]$, we define the constant 
$$M\coloneqq \sqrt{\max_{1 \le p \le n'}\sum_{i=1}^n(\|x^i_{t_p}\|^2+1)},$$
based on the deterministic values $\{x^i_{t_p}\}_{p=1}^{n'}$, for $i = 1, \ldots, n$.
By \eqref{F_expectation_estimation}, we have
\begin{equation}
\label{M_bound}
    \EE \left[\frac{1}{n}\sup_{f \in {\Phi}}\Bigl|\sum_{i=1}^n\xi_i f(x_{t_p}^i)\Bigr|\right] \le \frac{A_3}{n}M, \;\forall \; p = 1, \ldots, n'.
\end{equation}

Using the Boole's inequality and \eqref{F_tail_prob}-\eqref{M_bound}, we have the following estimation of the tail probability, for any $a >0$,
\begin{align}
    &\mathbb{P}\left(\frac{1}{n}\max_{1\le p \le n'}\sup_{f \in {\Phi}}\Bigl|\sum_{i=1}^n\xi_i f(x_{t_p}^i)\Bigr| - \frac{A_3}{n}M \ge \frac{8\sqrt{\log n'}\max\{A_1,A_2\}}{n} M+ a\right)\\
    \le~& n' \max_{1 \le p \le n'}\mathbb{P}\left( \frac{1}{n}\sup_{f \in {\Phi}}\Bigl|\sum_{i=1}^n\xi_i f(x_{t_p}^i)  \Bigr| - \frac{A_3}{n}M \ge \frac{ 8\sqrt{\log n'}\max\{A_1,A_2\}}{n} M+ a\right)\notag\\
    \le~& n' \max_{1 \le p \le n'}\mathbb{P}\left( \frac{1}{n}\sup_{f \in {\Phi}}\Bigl|\sum_{i=1}^n\xi_i f(x_{t_p}^i)  \Bigr| - \EE \left[\frac{1}{n}\sup_{f \in {\Phi}}\Bigl|\sum_{i=1}^n\xi_i f(x_{t_p}^i)\Bigr|\right]\ge \frac{ 8\sqrt{\log n'}\max\{A_1,A_2\}}{n} M + a\right)\notag\\
    \le~& n'\exp\left(-\frac{n^2[\frac{ 8\sqrt{\log n'}\max\{A_1,A_2\}}{n} M + a]^2}{64\max\{A_1^2,A_2^2\} M^2}\right)\\
    \le~& n'\exp(-\log(n'))\exp\left(-\frac{n^2a^2}{64\max\{A_1^2,A_2^2\}M^2}\right) \\
    \le~&\exp\left(-\frac{n^2a^2}{64\max\{A_1^2,A_2^2\}M^2}\right).
\end{align}
Therefore, 
\begin{align}
    &\frac{1}{n}\EE \max_{1\le p \le n'}\sup_{f \in {\Phi}}\Bigl|\sum_{i=1}^n\xi_i f(x_{t_p}^i)\Bigr| \\
    \le~& \frac{A_3}{n}M + \frac{8\sqrt{\log n'}  \max\{A_1,A_2\}}{n}M + \int_{0}^{+\infty}\exp\left(-\frac{n^2a^2}{64\max\{A_1^2,A_2^2\}M^2}\right)\rmd a\\
    \le~&  \frac{A_3 + 8(\sqrt{\log n'} +1) \max\{A_1,A_2\}}{n}M \\
    =~& \frac{A_3 + 8(\sqrt{\log n'} +1) \max\{A_1,A_2\}}{n}\sqrt{\max_{1 \le p \le n'}\sum_{i=1}^n(\|x^i_{t_p}\|^2+1)}.
\end{align}
The above estimate is for deterministic functions in $C([0,T]; \RR^d)$ evaluated at time points $t_1, \ldots, t_{n'}$. Applying it to i.i.d. continuous stochastic processes $\{X^i_\cdot\}_{i=1}^n$ with the law $\mu$ and independent of $\xi_1, \ldots, \xi_n$, we obtain 
{\small
\begin{align}
    \frac{1}{n}\EE \max_{1 \le p \le n'}\sup_{f \in {\Phi}}\Bigl|\sum_{i=1}^n\xi_i f(X_{t_p}^i)\Bigr| &= \frac{1}{n}\EE\Bigl\{\EE \Bigl[\max_{1 \le p \le n'}\sup_{f \in {\Phi}}\Bigl|\sum_{i=1}^n\xi_i f(X_{t_p}^i)\Bigr|\,\big\vert\,X^1_\cdot,\dots,X^n_\cdot\Bigr]\Bigr\}\\
    &\le \frac{A_3+ 8(\sqrt{\log n'} + 1)\max\{A_1,A_2\}}{n} \EE \sqrt{\max_{1\le p \le n'}\sum_{i=1}^n(\|X_{t_p}^i\|^2 + 1)}\notag\\
    &\le \frac{A_3 + 8(\sqrt{\log n'} + 1)\max\{A_1,A_2\}}{\sqrt{n}}\left( \int_{C([0,T];\bR^d)}\sup_{0\le t \le T}\|x_t\|^2\rmd \mu(x) +1\right)^{\frac{1}{2}}.
\end{align}}
Using the above estimate and letting $t_0, t_1, \ldots, t_n$ define a partition of $[0,T]$, that is, $t_p= \frac{(p-1) T}{n'}$ for $p = 1,\dots,n'$, by Assumption~\ref{assump:module}, and with the notation $\Pi(t) = \lfloor\frac{n't}{T}\rfloor\frac{T}{n'}$, one can deduce
\begin{align}
    &\left|\frac{1}{n}\EE\sup_{t \in [0,T]}\sup_{f \in {\Phi}}\Bigl|\sum_{i=1}^n\xi_if(X_t^i)\Bigr| - \frac{1}{n}\EE\sup_{1\le p \le n'}\sup_{f \in {\Phi}} \Bigl|\sum_{i=1}^n\xi_i f(X_{t_p}^i)\Bigr|\right| \notag \\
    \le~&\frac{1}{n}\EE\sup_{t \in [0,T]}\sup_{ f \in {\Phi}}\Bigl|\sum_{i=1}^n \xi_i \left(f(X_t^i) - f(X_{\Pi(t)}^i)\right)\Bigr| \notag \\
     \le~& A_1\EE \sup_{t \in [0,T]}\|X_t^1 - X_{\Pi(t)}^1\| \le A_1Q(\frac{T}{n'})^\alpha \log^\beta(2n').
\end{align}
Choosing $n' = \lfloor n^{\frac{1}{2\alpha}}\rfloor + 1$ gives
\begin{align}
    &\frac{1}{n}\EE\sup_{t \in [0,T]}\sup_{f \in {\Phi}}\Bigl|\sum_{i=1}^n\xi_if(X_t^i)\Bigr|\\
    \le& \left[\frac{A_3 + 8\max\{A_1,A_2\}}{\sqrt{n}} + 8\max\{A_1,A_2\}\sqrt{\frac{\log(2^{2\alpha}n)}{2\alpha n}}\right]\left( \int_{C([0,T];\bR^d)}\sup_{0\le t \le T}\|x_t\|^2\rmd \mu(x) +1\right)^{\frac{1}{2}} \\
    &\quad +\frac{A_1Q T^\alpha}{\sqrt{n}}\left(\frac{\log(2^{4\alpha}n)}{2\alpha}\right)^{\beta}.
\end{align}
Therefore we have proved \eqref{eq:EDprocess}. 
\end{proof}

The next theorem shows that the law of the stochastic differential equation (SDE) satisfies the condition of Theorem \ref{thm: process}, {\it i.e.}, Assumption~\ref{assump:module}.

\begin{thm}\label{thm: prior_estimation}
Given a constant $T > 0$, a complete filtered probability space $(\Omega,\mathcal{F},\mathbb{F} = \{\mathcal{F}_t\}_{0\le t \le T},\mathbb{P})$ supporting an m-dimensional Brownian motion $W$ as well as an $\mathcal{F}_0$-measurable $\bR^d$-valued random variable $\eta$. We consider the following SDE
\begin{equation}
    \rmd X_t = B(t,X_t) \rmd t + \Sigma(t,X_t)\rmd W_t, \quad X_0 = \eta,
\end{equation}
where $B:[0,T]\times \bR^d \rightarrow \bR^d$ and $\Sigma: [0,T]\times \bR^d \rightarrow \bR^{d \times m}$ satisfy: $\forall t \in [0,T], \; x, x' \in \RR^d$,
\begin{align}
    \|B(t,x) - B(t,x')\|^2 + \|\Sigma(t,x) - \Sigma(t,x')\|_F^2 &\le K^2\|x - x'\|^2, \notag\\
    \|B(t,0)\| + \|\Sigma(t,0)\|_F &\le K, \qquad \qquad 
\end{align}
with $K$ being a positive constant and $\|\cdot\|_F$ denoting the Frobenius norm on $\bR^{d\times m}$. It is well-known that the above SDE admits a unique strong solution ({\it cf.} \cite[Theorem~3.3.1]{zhang2017backward}). We denote by $\mu_0 := \MCL(\eta)$, $\mu := \MCL(X)$ the laws of $\eta$ and $X$, respectively.

\begin{enumerate}
    \item[(a)] Assume that $\EE\|\eta\|^2 \le K^2 $, then there exists a positive constant $C$ depending only on $K$ and $T$, such that
    \begin{equation}
        \EE \sup_{0\le t \le T}\|X_t\|^2 \le C,
    \end{equation}
    and
     \begin{equation}
        \int_{C([0,T];\bR^d)}[\Delta(x,h)]^2\rmd \mu(x) \le C h\log\left(\frac{2T}{h}\right). 
    \end{equation}
    \item[(b)] Assume that 
    \begin{equation}
        \mathcal{W}_1^2(\mu_0,\tilde{\mu}) \le 2K^2\mathcal{H}(\tilde{\mu}|\mu_0), \; \forall \tilde{\mu} \ll \mu_0,
    \end{equation}
    and
    \begin{equation}
        \sup_{t\in[0,T],x\in\bR^d}\|\Sigma(t,x)\|_F \le K.
    \end{equation}
    
    Then, there exists a positive constant $C$ depending on $K$ and $T$, such that
    \begin{equation}
        \mathcal{W}_1^2(\mu,\tilde{\mu}) \le C\mathcal{H}(\tilde{\mu}|\mu), \;\forall \tilde{\mu} \ll \mu,
    \end{equation}
    where $\mathcal{H}$ is the relative entropy defined in Theorem \ref{thm: iid}~(d). 
\end{enumerate}
\end{thm}
\begin{proof}
We defer the proof to \ref{app:prior_estimation} as it is less relevant to the main object of this paper. As mentioned in Remark \ref{rem_process_continuity}, the second part of claim (a) has been established in \cite{fischer2009moments} for more general It\^{o} processes. However, our estimates show that the constant $C$ does not depend on the dimensions $d$ and $m$.
\end{proof}

\section{Examples of Classes of Functions Satisfying Assumption~\ref{function_class_assumption}}\label{sec:testfunction}


\subsection{The Reproducing Kernel Hilbert Spaces (RKHSs)}
RKHS has developed into an essential tool in many areas, especially statistics and machine learning \cite{hofmann2008kernel}. We first recall its definition: a Hilbert space of functions $f: \RR^d \to \RR$, is said to be an RKHS if all evaluation functionals are bounded and linear. A more intuitive definition is through the so-called reproducing kernel. 

A symmetric function $k:\bR^d \times \bR^d \rightarrow \bR$ is called a positive kernel function on $\RR^d$, if
    \begin{equation}\label{positive kernel}
        \sum_{i=1}^n\sum_{j=1}^n a_ia_jk(x^i,x^j) \ge 0
    \end{equation}
holds for any $n \in \mathbb{N}, x^1,\dots,x^n \in \bR^d$, and $a_1,\dots,a_n \in \bR$. We then define the inner product space
\begin{equation}
    \{f(x) = \sum_{i=1}^m \alpha_i k(x,x_i): m \in \mathbb{N}^+, \{\alpha_i\}_{i=1}^m \subset \bR^m, \{x_i\}_{i=1}^m \subset \bR^{dm}\}
\end{equation}
with the inner product
\begin{equation}
    \langle f, g\rangle_{\mathcal{H}_k} = \sum_{i=1}^{m^f}\sum_{j=1}^{m^g}\alpha_i^f\alpha_j^fk(x_i^f,x_j^g), \; \forall f(x) = \sum_{i=1}^{m^f}\alpha_i^f k(x,x_i^f), \; g(x) = \sum_{i=1}^{m^g}\alpha_i^gk(x,x_i^g).
\end{equation}
The reproducing kernel Hilbert space $\mathcal{H}_k$ is the completion of the inner product space with respect to $\|\cdot\|_{\mathcal{H}_k} = \sqrt{\langle \cdot,\cdot\rangle_{\mathcal{H}_k}}$. 
Moreover, $\mathcal{H}_k$ satisfies the reproducing property:
\begin{equation}\label{def:reproducing property}
    f(x) = \langle f,k(x,\cdot)\rangle_{\mathcal{H}_k}, \forall x \in \bR^d, f \in \mathcal{H}_k.
\end{equation}
In particular, the function $k$, called the reproducing kernel of $\mathcal{H}_k$, satisfies 
\begin{equation}
    k(x,y) = \langle k(x,\cdot), k(y,\cdot)\rangle_{\mathcal{H}_k}.
\end{equation}
We refer the interested readers to \cite{aronszajn1950theory} for more properties of RKHSs. Theorem~\ref{thm: rkhs i.i.d.} below guarantees that RKHSs associated with Gaussian kernels, Laplacian kernels and neural tangent kernels satisfy Assumption \ref{function_class_assumption}. We remark that the choice in Theorem~\ref{thm: rkhs i.i.d.} reproduces the maximum mean discrepancy in \cite{borgwardt2006integrating}.

\begin{thm}\label{thm: rkhs i.i.d.}
Assume the reproducing kernel $k(\cdot, \cdot)$ satisfies:
\begin{enumerate}
    \item[(a)] There exist constants $K_1, K_2>0$, such that $\forall x,y \in \bR^d$,  $k(x,x) + k(y,y) - 2k(x,y) \le K_1^2 \|x-y\|^2$ and $K_2  = \sqrt{k(0,0)}$;
    \item[(b)] 
    If $\mu$ is a signed measure on $\bR^d$,
    \begin{equation}
        \int_{\bR^d}k(x,y)\rmd\mu(y) = 0, \;\forall x \in \bR^d \Rightarrow \mu \equiv 0. 
    \end{equation}
\end{enumerate}
Then for any $\mu \in \mathcal{P}^1(\bR^d)$, ${\Phi} = \{f \in \mathcal{H}_k, \|f\|_{\mathcal{H}_k} \le 1\}$ satisfies Assumption \ref{function_class_assumption} with $A_1 = K_1$, $A_2 = K_2$  and $A_3 = \sqrt{2} \max\{K_1,K_2\}$.
\end{thm}

\begin{proof}
By definition, $k(x,\cdot) \in \mathcal{H}_k$ for all $x \in \bR^d$. Hence, Assumption~\ref{function_class_assumption}~(a) is implied by item (b) above.
For Assumption~\ref{function_class_assumption} (b), $\forall f \in \mathcal{H}_k$ such that  $\|f\|_{\mathcal{H}_k} \le 1$, we compute
\begin{align}
    |f(x) - f(y)| &=|\langle f,k(x,\cdot) - k(y,\cdot)\rangle_{\mathcal{H}_k}| 
    \le \sqrt{\langle k(x,\cdot) - k(y,\cdot),k(x,\cdot) - k(y,\cdot)\rangle_{\mathcal{H}_k}} \notag \\&= \sqrt{k(x,x) + k(y,y) - 2k(x,y)} \le K_1\|x - y\|,
\end{align}
for any $x,y \in \bR^d$, which implies that the Lipschitz constant is $K_1$. 

For Assumption~\ref{function_class_assumption} (c), we first derive an estimate for $k(x,x)$. To this end, let $n = 2$, $x^1 = x$, $x^2 = 0$ in inequality \eqref{positive kernel}, we have
\begin{equation}
    a_1^2k(x,x) + a_2^2k(0,0) + 2a_1a_2k(x,0) \ge 0,
\end{equation}
for any $a_1,a_2$, implying $|k(x,0)| \le \sqrt{k(x,x)k(0,0)}$.
Therefore,
\begin{equation}
    (\sqrt{k(x,x)}-\sqrt{k(0,0)})^2 = k(x,x) + k(0,0) - 2\sqrt{k(x,x)k(0,0)} \le k(x,x) + k(0,0) - 2k(x,0) \le K_1^2\|x\|^2,
\end{equation}
and $k(x,x) \le (\sqrt{k(0,0)} + K_1 \|x\|)^2 = (K_2 + K_1\|x\|)^2 \le 2(K_2^2 + K_1^2\|x\|^2)$.
Now we estimate the Rademacher complexity of ${\Phi} = \{f \in \mathcal{H}_k, \|f\|_{\mathcal{H}_k} \le 1\}$:
\begin{align}
    \mathrm{Rad}_n({\Phi},\mathcal{X}) &= \frac{1}{n}\EE  \sup_{\|f\|_{\mathcal{H}_k} \le 1} \Bigl|\sum_{i=1}^n\xi_i f(x^i)\Bigr| =\frac{1}{n}\EE  \sup_{\|f\|_{\mathcal{H}_k} \le 1} \Bigl|\sum_{i=1}^n\xi_i \langle f,k(x^i,\cdot)\rangle_{\mathcal{H}_k}\Bigr| \\
    &=\frac{1}{n}\EE  \sup_{\|f\|_{\mathcal{H}_k} \le 1} | \langle f,\sum_{i=1}^n\xi_i k(x^i,\cdot)\rangle_{\mathcal{H}_k}| \le \frac{1}{n}\EE  \sqrt{\sum_{i=1}^n\sum_{j=1}^n\xi_i \xi_jk(x^i,x^j)} \notag \\
    &\le \frac{1}{n} \sqrt{\EE \sum_{i=1}^n\sum_{j=1}^n\xi_i \xi_jk(x^i,x^j)}= \frac{1}{n} \sqrt{\sum_{i=1}^n k(x^i,x^i)} \le \frac{1}{n}\sqrt{2\sum_{i=1}^n (K_1^2\|x^i\|^2 + K_2^2)} \\
    &\le \frac{A_3}{n}\sqrt{(\sum_{i=1}^n\|x^i\|^2 + 1)}.
\end{align}
\end{proof}

\subsection{The Barron Space}\label{sec:barron}

Barron space was firstly introduced in \cite{weinan2019barron,ma2019priori}, which is designed to analyze the approximation and generalization properties of two-layer neural networks. It can be considered as the continuum analog of two-layer neural networks. See \cite[Section 2.1]{weinan2019barron} for a detailed discussion on Barron space.
\begin{defn}
We say $f: \bR^d \rightarrow \bR$ is a Barron function, if $f$ admits the following representation:

\begin{equation}\label{def_Barron_Space}
    f(x) = \int_{\mathbb{S}^{d}} \sigma(\omega\cdot x + b) \rmd\rho(\omega,b),
\end{equation}
where $\sigma(x) = \max\{x,0\}$ is the ReLU function, and $\rho$ is a finite signed measure on $\mathbb{S}^{d} = \{(\omega,b)\in \RR^{d+1}, \|\omega\|^2 + |b|^2 = 1\}$ with $\|\cdot\|$ being the Euclidean norm. We will use the Barron space $\mathcal{B}$ to denote the collection of all Barron functions and define a norm $\|\cdot\|_\mathcal{B}$ on the Barron space as follows:
\begin{equation}
    \|f\|_\mathcal{B} = \inf_{\rho} \|\rho\|_{TV},
\end{equation}
where the infimum is taken over all $\rho$ for which \eqref{def_Barron_Space} holds for all $x \in \bR^d$, and $\|\cdot\|_{TV}$ is the total variation of $\rho$.
\end{defn}

The following theorem reveals some useful properties of Barron space and shows that the unit ball of Barron space, denoted by $\mathcal{B}_1 := \{f \in \mathcal{B}, \|f\|_\mathcal{B} \leq 1\}$, can serve as a good choice of the test function class ${\Phi}$. Although Theorem~\ref{thm: Property_Barron_Space} is stated when the activation function $\sigma$ is the ReLU function, all claims except (d) hold if $\sigma$ is any 1-Lipschitz, nonlinear function and hence $\mathcal{B}_1$ corresponding to such $\sigma$ is also a good choice of the test function class $\Phi$. The proof of claim (e) for general activation functions can be found in \cite{pinkus1999approximation} while other claims (a), (b), (c) and (f) follow the same proof presented below.
\begin{thm}\label{thm: Property_Barron_Space}

\begin{enumerate}
    \item[(a)] The Barron space is a Banach space.
    \item[(b)] For any $f \in \mathcal{B}$, f is Lipschitz continuous with the Euclidean norm in $\bR^d$ and $\mathrm{Lip}(f) \le \|f\|_\mathcal{B}$. 
    \item[(c)] Denote by $\mathcal{P}(\mathbb{S}^d)$ all probability measures on $\mathbb{S}^d$. For any $\pi \in \mathcal{P}(\mathbb{S}^d)$, let $k_\pi(x, x'): \RR^d \times \RR^d \to \RR$ be
    \begin{equation}
        k_\pi(x,x') = \int_{\mathbb{S}^d}\sigma(\omega \cdot x + b)\sigma(\omega\cdot x'+b)\rmd \pi(\omega,b),
    \end{equation}
    and $\mathcal{H}_\pi$ be the RKHS associated with $k_\pi$. Then,
    \begin{equation}
        \mathcal{B} = \mathop{\cup}_{\pi \in \mathcal{P}(\mathbb{S}^d)} \mathcal{H}_\pi,
    \end{equation}
    and 
    \begin{equation}\label{def:BarronNorm}
        \|f\|_\mathcal{B} = \inf_{\pi \in \mathcal{P}(\mathbb{S}^d)}\|f\|_{\mathcal{H}_\pi},
    \end{equation}
    with $\|\cdot\|_{\mathcal{H}_\pi}$ being the norm in $\mathcal{H}_\pi$.
    \item[(d)] Let $f$ be a measurable function in $L^2(\bR^d)$ and 
    \begin{equation}
        \gamma(f) = \int_{\bR^d}(1+\|\omega\|^2)|\hat{f}(\omega)|\rmd \omega < +\infty,
    \end{equation}
    where $\hat{f}(\omega) = \int_{\bR^d}f(x)e^{-i\omega\cdot x} \rmd x$ is the Fourier transform of f. Then $f \in \mathcal{B}$ and 
    \begin{equation}
        \|f\|_\mathcal{B} \le 4\gamma(f).
    \end{equation}
    \item[(e)] For any compact set $K \subset \bR^d$, the restriction of $\mathcal{B}$ on $K$ is dense in $C(K)$.
    \item[(f)] Let $\mathcal{X} = \{x^1,\dots,x^n\}$, where $x^i \in \bR^d$ and ${\Phi} = \mathcal{B}_1$. Then the empirical Rademacher complexity satisfies
    \begin{equation}
        \mathrm{Rad}_n({\Phi},\mathcal{X}) \le \frac{2}{n}\sqrt{\sum_{i=1}^n(\|x^i\|^2+1)},
    \end{equation}
    and ${\Phi} = \mathcal{B}_1$ satisfies Assumption~\ref{function_class_assumption} with $A_1 = A_2 = 1$ and $A_3 =2$.
\end{enumerate}

\end{thm}
\begin{rem}
In \cite[Section~9]{barron1993universal}, examples of functions with bounded $\gamma(f)$ are provided ({\it e.g.}, Gaussian, positive definite functions, linear functions, radial functions and functions in fractional Sobolev spaces $H^s(\bR^d)$ with $s > \frac{d}{2} + 1$). By claim (d), they all belong to the Barron space. 
\end{rem}

\begin{proof}
\begin{enumerate}
    \item[(a)]  Let $\mathcal{M}$ be the set of all signed measures on $\mathbb{S}^d$. With $\|\rho\|_\mathcal{M} := \|\rho\|_{TV}$,  $\mathcal{M}$ is a Banach space. Let $\mathcal{N} = \{\rho \in \mathcal{M}: \int_{\mathbb{S}^d} \sigma(\omega\cdot x + b) \rmd \rho(\omega,b) = 0, \; \forall x \in \bR^d\}$, then $\mathcal{N}$ is a closed subspace of $\mathcal{M}$. It is evident that $\mathcal{B} \cong \mathcal{M}/ \mathcal{N}$, thus we obtain the desired result \cite[Theorem 1.41]{rudin1991functional}.
    
    \item[(b)] For any $(\omega,b) \in \mathbb{S}^d$, $\mathrm{Lip}(\sigma(\omega\cdot x + b) )\le 1$. Then our result follows from the subadditivity of $\mathrm{Lip}$.

    \item[(c)] See \cite[Theorem 3]{weinan2019barron}.
    
    \item[(d)] Our proof is similar to that in \cite[Theorem 6]{klusowski2016risk}. By the property of Fourier transform, one has
    \begin{equation}\label{Fourier_relation}
        f(x) = x\cdot \nabla f(0) + f(0) + \int_{\bR^d}(e^{i \omega\cdot x} - i\omega\cdot x - 1)\hat{f}(\omega)\rmd \omega.
    \end{equation}
    For $|z| \le c$, the following identity holds
    \begin{equation}
        -\int_{0}^{c}[\sigma(z-u)e^{iu} + \sigma(-z-u)e^{-iu}] \rmd u = e^{iz} - iz - 1.
    \end{equation}
    The above two equations can be found in \cite[Theorem~6]{klusowski2016risk}.
    Now let $c = \|\omega\|$, $z = \omega \cdot x$, $\alpha(\omega) = \omega/\|\omega\|$ and $u = \|\omega\|t$, we have:
    \begin{equation}
        -\|\omega\|^2 \int_{0}^{1}\sqrt{1+t^2}\left[\sigma\left(\frac{\alpha(\omega)\cdot x - t}{\sqrt{1+t^2}}\right) e^{i\|\omega\|t} +\sigma\left(\frac{-\alpha(\omega)\cdot x-t}{\sqrt{1+t^2}}\right)e^{-i\|\omega\|t}\right]\rmd t = e^{i\omega\cdot x } - i\omega\cdot x -1.
    \end{equation}
    In other words, 
    \begin{equation}
        e^{i\omega\cdot x } - i\omega\cdot x -1 = \int_{\mathbb{S}^d}\sigma(\omega'\cdot x + b)\rmd \rho_{\omega} (\omega',b),
    \end{equation}
    where
    \begin{equation}
        \rho_\omega = -\|\omega\|^2 \int_{0}^1\sqrt{1+t^2} [e^{i\|\omega\|t}\delta_{(\frac{\alpha(\omega)}{\sqrt{1+t^2}},-\frac{t}{\sqrt{1+t^2}})} + e^{-i\|\omega\|t}\delta_{(-\frac{\alpha(\omega)}{\sqrt{1+t^2}},-\frac{t}{\sqrt{1+t^2}})}]\rmd t.
    \end{equation}
    Hence, $e^{i\omega\cdot x} - i\omega\cdot x-1 $ is in the Barron space $\mathcal{B}$ with $\|\rho_{\omega}\|_{TV} \le 2\sqrt{2}\|\omega\|^2$. Recall that equation \eqref{Fourier_relation} gives
    \begin{align}
        f(x) &=\|\nabla f(0)\|\left[ \sigma\left(\frac{\nabla f(0) \cdot x}{\|\nabla f(0)\|}\right) + \sigma\left(-\frac{\nabla f(0) \cdot x}{\|\nabla f(0)\|}\right)\right] +f(0)\sigma(0\cdot x + 1)\notag \\
        & \qquad +  \int_{\bR^d} \int_{\mathbb{S}^d}\sigma(\omega'\cdot x + b)\rmd \rho_{\omega}(\omega',b) \hat{f}(\omega)\rmd \omega,
    \end{align}
    and one concludes
    \begin{align}
        \|f\|_\mathcal{B} \le 2\|\nabla f(0)\| + |f(0)| + 2\sqrt{2}\int_{\bR^d}\|\omega\|^2|\hat{f}(\omega)|\rmd \omega \\
        \leq  \int_{\bR^d}[1 + 2\|\omega\| + 2\sqrt{2}\|\omega\|^2]|\hat{f}(\omega)|\rmd \omega \le 4\gamma(f).
    \end{align}
    \item[(e)] For any $f \in C_0^\infty(\bR^d)$, we know that $f$ is a Schwartz function. Therefore, $\hat{f}$ is also a Schwartz function and we have $\gamma(f) < +\infty$. Hence, $f \in \mathcal{B}$; see e.g. \cite[Section~6]{stein2011fourier}. Because the restriction of $C_0^\infty(\bR^d)$ on $K$ is dense in $C(K)$, we easily conclude.
    \item[(f)] Our proof here is similar to the one in \cite[Theorem 6]{weinan2019barron}.
    {\small \begin{align}
        \mathrm{Rad}_{n}({\Phi},\mathcal{X}) &= \frac{1}{n}\EE \sup_{f \in {\Phi}}\sum_{i=1}^n\xi_i f(x^i) \notag \\
        &= \frac{1}{n}\EE \sup_{\|\rho\|_{TV} \le 1}\Bigl|\int_{\mathbb{S}^d}\sum_{i=1}^n\xi_i\sigma(\omega \cdot x^i +b)\rmd \rho(\omega,b)\Bigr|\notag \\
        &= \frac{1}{n}\EE \sup_{(\omega,b) \in \mathbb{S}^d}\Bigl|\sum_{i=1}^n\xi_i\sigma(\omega \cdot x^i + b)\Bigr|\notag \\
        &\le \frac{1}{n}\left[\EE \max \Big\{\sup_{(\omega,b)\in \mathbb{S}^d}\sum_{i=1}^n\xi_i\sigma(\omega \cdot x^i + b), 0\Big\} + \EE \max \Big\{\sup_{(\omega,b)\in \mathbb{S}^d}-\sum_{i=1}^n\xi_i\sigma(\omega \cdot x^i + b), 0\Big\}\right] \notag \\
        &= \frac{2}{n}\EE \sup_{(\omega,b)\in \mathbb{S}^d}\sum_{i=1}^n\xi_i\sigma(\omega \cdot x^i + b),
    \end{align}}
where the last equality holds due to  $\sup_{(\omega,b)\in \mathbb{S}^d}\sum_{i=1}^n\xi_i\sigma(\omega \cdot x^i + b) \ge 0$ and the symmetry of $\xi_1,\dots,\xi_n$. Then Lemma 26.9 in \cite{shalev2014understanding} gives
\begin{align}
    \mathrm{Rad}_n({\Phi},\mathcal{X})&\le \frac{2}{n} \EE \sup_{(\omega,b)\in \mathbb{S}^d}\sum_{i=1}^n\xi_i\sigma(\omega\cdot x^i+b) \leq \frac{2}{n} \EE \Big\| \sum_{i=1}^n\xi_i((x^i)\transpose,1)\transpose\Big\|\notag \\
    &\le \frac{2}{n} \sqrt{\EE\Big\|\sum_{i=1}^n\xi_i((x^i)\transpose,1)\transpose\Big\|^2}\le \frac{2}{n}\sqrt{\sum_{i=1}^n(\|x^i\|^2 + 1)}. 
\end{align}
Finally, Assumption~\ref{function_class_assumption} (a) is fulfilled by claim (e) above, and  claims (b) and the above estimate together imply that ${\Phi}$ satisfies Assumption~\ref{function_class_assumption} (b)--(c) with $A_1 = A_2 = 1$ and $A_3 =2$.
\end{enumerate}

\end{proof}

\begin{rem}
As the Barron space serves as the continuum analog of two-layer neural networks, the generalized Barron space or the Banach space associated with multi-layer networks introduced in \cite{wojtowytsch2020banach} serves as the continuum analog of multi-layer neural networks. By the fact that the unit ball of the Barron space $\mathcal{B}_1$ is a Polish space, one can define a signed Radon measure $\rho$ on the Borel $\sigma$-algebra of $\mathcal{B}_1$, and then define 
\begin{align}
    f_\rho(x) &= \int_{\mathcal{B}_1}\sigma(g(x))\rmd \rho(g), \notag \\
    \|f\|_{\mathcal{B}^2} &= \inf \{\|\rho\|_{TV}: f = f_\rho \text{ on }\bR^d\},\notag \\
    \mathcal{B}^2 &= \{ f \in C(\bR^d),  \|f\|_{\mathcal{B}^2} < +\infty\},
\end{align}
where the integral is in the sense of Bochner intergrals. $\mathcal{B}^2$ can be interpreted as the Banach space associated with three-layer neural networks. We can then repeat this process and define $\mathcal{B}^L$ for any $L \ge 2$, which is associated with $(L+1)$-layer neural networks. Naturally we denote by $\mathcal{B}^1$ the Barron space  $\mathcal{B}$ defined via \eqref{def_Barron_Space}.  
It can be proven that the unit ball of the generalized Barron space is also a suitable test function class. For more technical issues and intuition about $\mathcal{B}^L$, we refer to \cite[Section 2.2]{wojtowytsch2020banach}.
\end{rem}

\subsection{Flow-induced Function Spaces\footnote{It is a class of function mimicking the limiting behavior of residual neural networks. And the name ``flow-induced function class'' would be more appropriate as the terminology ``space'' implies linearity, which we currently are not able to prove. Nevertheless, we keep the title as it is, following the terminology introduced in \cite{weinan2019barron}.}}

Flow-induced function spaces introduced in \cite{weinan2019barron} can serve as a continuum analog as the residual neural networks (ResNet,~\cite{he2016deep}). Our definition here is slightly different from the original definition in \cite{weinan2019barron}. Such an alteration will enable us to bound the  Rademacher complexity by 
\begin{equation}
    \frac{1}{n}\sqrt{\sum_{i=1}^n(\|x^i\|^2 + 1)}
\end{equation}
as requested by Assumption \ref{function_class_assumption}, while the choice in \cite{weinan2019barron} yields a bound 
\begin{equation}
    \sqrt{\frac{1+ \max_{1 \le i \le n}\|x^i\|_\infty}{n}}.
\end{equation}
Given an integer $D \ge d+1$, let $\rho = \{\rho_\tau\}_{0\le \tau\le 1}$ be a class of vector signed measures from $\mathbb{S}^{D-1} = \{\omega \in \bR^D,\|\omega\| = 1\}$ to $\bR^D$ such that the following ordinary differential equation (ODE)
\begin{equation}\label{ODE:cfs}
    \frac{\rmd Z_\tau}{\rmd \tau} = \int_{\mathbb{S}^{D-1}}\sigma(\omega \cdot Z_\tau)\rmd \rho_\tau(\omega), \quad \forall\; 0 \le \tau \le 1, \quad Z_0 = z_0
\end{equation}
with $\sigma$ being the ReLU function, is well-posed for any initial condition $z_0$, {\it i.e.}, the solution exists, is unique, and is continuous with respect to the initial condition $z_0$. Let us denote by $\Psi$ the collection of all admissible $\rho$ satisfying this well-posedness condition.
If the ODE \eqref{ODE:cfs} is discretized by the Euler scheme in time and $\rho_\tau$ is discrete, then the system gives the same structure as ResNet, a widely used deep neural network architecture.

For any $x \in \bR^d$, let $Z_{\rho}(\tau,x)$ be the solution of \eqref{ODE:cfs} with the initial condition $z_0 = (x\transpose,1,\bm{0}_{(D-d-1)})\transpose$, where $\bm{0}_{(D-d-1)}$ denotes a vector of zeros of length $D-d-1$.
And for any $\rho = \{\rho_\tau\}_{0\le \tau \le 1} \in \Psi$, we define
\begin{equation}
 \Lambda(\rho,\tau) = \sqrt{\sum_{i=1}^D\|\rho_\tau^i\|_{TV}^2},
\end{equation}
where $\rho_\tau^i$ is the $i^{th}$ component of $\rho_\tau$. To simplify the discussion, we in addition require that $\Lambda(\rho,\tau)$ is continuous with respect to $\tau$ for any $\rho \in \Psi$.
Then, we can define the flow-induced function spaces as the space $\mathcal{D}$ given by the class of functions $f$ admitting the representation $f = f_{\rho, \alpha}$ where
\begin{align}
    f_{\rho,\alpha}(x) &:= \alpha\transpose Z_\rho(1,x), \quad \forall \rho \in \Psi, \;  \alpha \in \bR^D, \notag \\
    \|f\|_{\mathcal{D}} &:= \inf \left\{\|\alpha\|\exp\left(\int_{0}^1\Lambda(\rho,\tau)\rmd \tau\right), \; f = f_{\rho,\alpha} \text{ on } \bR^d\right\}, \notag \\
    \mathcal{D}&:= \{ f \in C(\bR^d), \|f\|_\mathcal{D} < +\infty\}.
\end{align}
The following theorem gives some useful properties of the flow-induced function spaces and indicates that the unit balls of the flow-induced function spaces are also appropriate test function classes ${\Phi}$.
\begin{thm}\label{thm: compositional}
Fix an integer $D \ge d+1$.
\begin{enumerate}
    \item[(a)] For any $f \in \mathcal{D}$, $\mathrm{Lip}(f) \le \|f\|_\mathcal{D}$ and $|f(0)| \le \|f\|_\mathcal{D}$.
    \item[(b)] If $D \ge d+2$, then for any $f \in \mathcal{B}$, $f \in \mathcal{D}$ and $\|f\|_\mathcal{D} \le e\|f\|_\mathcal{B}$. Hence, by Theorem  \ref{thm: Property_Barron_Space} (e), the restriction of $\mathcal{D}$ on any compact set $K \subset \bR^d$ is dense in $C(K)$.
    \item[(c)] Let $\mathcal{X} = \{x^1,\dots,x^n\}$, where $x^i \in \bR^d$ and ${\Phi} =\{f \in \mathcal{D}, \|f\|_\mathcal{D} \le 1\}$. Then, the empirical Rademacher complexity satisfies
    \begin{equation}
        \mathrm{Rad}_n({\Phi},\mathcal{X}) \le \frac{e^2}{n}\sqrt{\sum_{i=1}^n(\|x^i\|^2 + 1)},
    \end{equation}
    and ${\Phi}$ satisfies the Assumption~\ref{function_class_assumption} with $A_1 = A_2 = 1$ and $A_3 = e^2$.
\end{enumerate}
\end{thm}
\begin{proof}
For claim (a), given $\rho \in \Psi$, let $Z_\tau$, $Z_\tau'$ be the solutions to the ODE \eqref{ODE:cfs} with the initial conditions $z_0, z_0' \in \bR^D$, respectively. Then,
\begin{align}
    \frac{\rmd}{\rmd \tau}\|Z_\tau - Z_\tau'\|^2 &= 2\int_{\mathbb{S}^{D-1}}[\sigma(\omega\cdot Z_\tau) - \sigma(\omega \cdot Z_\tau')]\rmd [Z_\tau - Z_\tau']\transpose \rho_\tau(\omega) \notag \\
    &\le 2 \int_{\mathbb{S}^{D-1}} |\omega\cdot(Z_\tau - Z_\tau')| \rmd |[Z_\tau - Z_\tau']\transpose \rho_\tau(\omega)|\notag \\
    &\le 2 \Lambda(\rho,\tau)\|Z_\tau - Z_\tau'\|^2,
\end{align}
where $|\cdot|$ is obtained by taking element-wise absolute values of a vector. Integrating both sides from $\tau = 0$ to 1 gives
\begin{equation}
    \|Z_1 - Z_1'\| \le \|z_0 - z_0'\| \exp\left(\int_{0}^1\Lambda(\rho,\tau)\rmd \tau\right).
\end{equation}
Taking $z_0 = (x\transpose,1,\bm{0}_{(D-d-1)})$ and $z_0' = ((x')\transpose,1, \bm{0}_{(D-d-1)})$, we deduce
\begin{equation}
    |\alpha\transpose Z_\rho(1,x) - \alpha\transpose Z_\rho(1,x')| \le \|\alpha\| \exp\left(\int_{0}^1\Lambda(\rho,\tau)\rmd \tau\right)\|x- x'\|,
\end{equation}
and by taking $z_0 = (\bm{0}_d,1,\bm{0}_{(D-d-1)})\transpose$ and $z_0' = \bm{0}_D\transpose$, we have
\begin{equation}
    |\alpha\transpose Z_\rho(1,0)| \le \|\alpha\|\exp\left(\int_{0}^1\Lambda(\rho,\tau)\rmd \tau\right).
\end{equation}
Now by the definition of $f$ in  $\mathcal{D}$ and $\|f\|_\mathcal{D}$, we easily conclude.

For claim (b), without loss of generality, assume $D = d+2$. For any $f \in \mathcal{B}$ and $\epsilon > 0$, let
\begin{equation}
    f(x) = \int_{\mathbb{S}^d}\sigma(\omega \cdot x + b)\rmd \tilde{\rho}(\omega,b)
\end{equation}
with $0 < \|\tilde{\rho}\|_{TV}\le \|f\|_\mathcal{B} + \epsilon $. 
Define $\tilde{\rho}$:
\begin{equation}
    \hat{\rho}_\tau^i = 0, \quad 1 \le i \le d+1, \text{ and } \hat{\rho}_\tau^{d+2} = \frac{\tilde{\rho}}{\|\tilde{\rho}\|_{TV}}, \text{ for } 0 \leq \tau \leq 1.
\end{equation}
Then it is easy to check that $Z_{\hat{\rho}}(\tau,x)= (x\transpose,1,\frac{\tau f(x)}{\|\tilde{\rho}\|_{TV}})\transpose.$
and hence $f(x) = \|\tilde{\rho}\|_{TV}\, \bm{1}_{d+2}\transpose \, Z_{\hat{\rho}}(1,x)$ where $\bm{1}_{d+2} = (\bm{0}_{d+1},1)\transpose$.
Combining with the fact that
\begin{equation}
    \Lambda(\hat{\rho},\tau) \equiv 1,  \quad 0 \leq \tau \leq 1,
\end{equation}
we know that
\begin{equation}
    \|f\|_\mathcal{D} \le e \|\tilde{\rho}\|_{TV} \le e[\|f\|_\mathcal{B} + \epsilon],
\end{equation}
for any $\epsilon > 0$, which concludes our proof.

For claim (c), we first prove that for any $f \in {\Phi} = \{ f \in \mc{D}, \|f\|_\mc{D} \leq 1\}$ and $\epsilon > 0$, there exist $\overline{\rho} \in \Psi$ and $\alpha \in \bR^D$ such that $f(x) = \alpha\transpose Z_{\overline{\rho}}(1,x)$ with
\begin{equation}\label{control of derivative}
    \|\alpha\|\exp\left(\int_{0}^1\Lambda(\overline{\rho},\tau)\rmd \tau\right) \le 1 + \epsilon, \quad \|\alpha\|\sup_{0\le \tau \le 1}\Lambda(\overline{\rho},\tau)\exp\left(\int_{0}^\tau\Lambda(\overline{\rho},\tau')\rmd \tau'\right) \le 1+\epsilon.
\end{equation}
First, we choose $\alpha \in \bR^D$ and $\rho \in \Psi$ satisfying that
\begin{equation}
    f(x) = f_{\rho,\alpha}(x) = \alpha\transpose Z_{\rho}(1,x), \text{ and } \|\alpha\|\exp\left(\int_{0}^1\Lambda(\rho,\tau)\rmd \tau\right) \le 1 + \frac{\epsilon}{2} .
\end{equation}
Given a strictly increasing and continuously differentiable function  $F(\tau):[0,1]\rightarrow[0,1]$ satisfying that $F(0) = 0$, $F(1)  = 1$, we define
\begin{equation}
    \overline{\rho}_\tau := F'(\tau)\rho_{F(\tau)}.
\end{equation}
Then, from equation~\eqref{ODE:cfs} one has
\begin{equation}
    \frac{\rmd Z_{\rho}(F(\tau),x)}{\rmd \tau} = \int_{\mathbb{S}^{D-1}} \sigma(\omega\cdot Z_{\rho}(F(\tau),x))\rmd \overline{\rho}_\tau(\omega).
\end{equation}
In addition, $Z_{\overline{\rho}}(\tau,x) = Z_\rho(F(\tau),x)$ and $f_{\overline{\rho},\alpha} = f_{\rho,\alpha} = f$ on $\RR^d$. Note that the wellposedness of the above ODE can be deduced from the fact that $F$ is an isomorphism on $[0,1]$. Noticing that
\begin{equation}
    \int_{0}^1\Lambda(\overline{\rho},\tau)\rmd \tau = \int_{0}^1 F'(\tau)\Lambda(\rho,F(\tau))\rmd \tau = \int_{0}^1 \Lambda(\rho,\tau)\rmd \tau,
\end{equation}
we have
\begin{align}
    \|\alpha\|\exp\left(\int_{0}^1\Lambda(\overline{\rho},\tau)\rmd \tau\right) &= \|\alpha\|\exp\left(\int_{0}^1\Lambda(\rho,\tau)\rmd \tau\right) \le 1 + \half\epsilon, \notag \\
    \|\alpha\|\sup_{0\le \tau \le 1}\Lambda(\overline{\rho},\tau)\exp\left(\int_{0}^\tau\Lambda(\overline{\rho},\tau')\rmd \tau'\right) &=  \|\alpha\|\sup_{0\le \tau \le 1}F'(\tau)\Lambda(\rho,F(\tau))\exp\left(\int_{0}^{F(\tau)}\Lambda(\rho,\tau')\rmd \tau'\right). 
\end{align}
Next, assume without loss of generality that $\exp(\int_{0}^\tau\Lambda(\rho,\tau')\rmd \tau')$ strictly increases in $\tau$. Then, there exists a continuous and strictly increasing function $F^\ast:[0,1]\rightarrow [0,1]$ with $F^\ast(0) = 0$, $F^\ast(1) = 1$ and
\begin{equation}
    \exp\left(\int_{0}^{F^\ast(\tau)}\Lambda(\rho,\tau')\rmd \tau'\right) = 1- \tau + \tau \exp\left(\int_{0}^1\Lambda(\rho,\tau')\rmd \tau'\right),\quad \forall \tau\in[0,1],
\end{equation}
which gives (by implicit function theorem), when $\Lambda(\rho,F^\ast(\tau)) \neq 0$, 
\begin{equation}
    (F^\ast)'(\tau)\Lambda(\rho,F^\ast(\tau))\exp\left(\int_{0}^{F^\ast(\tau)}\Lambda(\rho, \tau')\rmd \tau'\right) = \exp\left(\int_{0}^1\Lambda(\rho,\tau')\rmd \tau'\right)-1.
\end{equation}
Then equation \eqref{control of derivative} is obtained by approximating $F^{*}$ through a continuously differentiable isomorphism on $[0,1]$.

Return to the proof of claim (c), for any $\epsilon > 0$, we define
\begin{multline}
    \mathfrak{P}^\epsilon = \bigg\{(\alpha,\rho) \in \bR^D \times \Psi: \\ \|\alpha\|\exp\left(\int_{0}^1\Lambda(\rho, \tau)\rmd \tau\right) \le 1+\epsilon, \, \|\alpha\|\sup_{0\le \tau \le 1}\Lambda(\rho,\tau)\exp\left(\int_{0}^1\Lambda(\rho,\tau')\rmd \tau'\right)\le 1 +\epsilon\bigg\},
\end{multline}
and 
\begin{equation}
    R_\tau^\epsilon = \frac{1}{n}\EE\sup_{(\alpha,\rho) \in \mathfrak{P}^\epsilon}\sum_{i=1}^n\xi_i \alpha\transpose Z_\rho(\tau,x^i).
\end{equation}
With \eqref{control of derivative}, we have $\mathrm{Rad}_n({\Phi},\mathcal{X})\le R_1^\eps$ for any $\eps$. So it suffices to deduce an upper bound for $R_1^\eps$.

A straightforward calculation gives
\begin{align}
|R_\tau^\epsilon - R_{\tau'}^\epsilon| &\le \sup_{(\alpha,\rho) \in \mathfrak{P}^\epsilon}\max_{1\le i \le n}|\alpha \transpose[Z_\rho(\tau,x^i) - Z_\rho(\tau',x^i)]\\
&\le \sup_{(\alpha,\rho) \in \mathfrak{P}^\epsilon}\max_{1\le i \le n}\int_{\tau}^{\tau'}\int_{\mathbb{S}^{D-1}}|\sigma(\omega\cdot Z_\rho(u,x^i))|\rmd |\alpha\transpose\rho_u|(\omega) \rmd u \\
&\le |\tau - \tau'| \sup_{(\alpha,\rho) \in \mathfrak{P}^\epsilon}\max_{1\le i \le n}\sup_{0 \le \tau \le 1}\|Z_\rho(\tau,x^i)\|\|\alpha\|\sup_{0 \le \tau \le 1}\Lambda(\rho,\tau) \\
&\le |\tau - \tau'|\sup_{(\alpha,\rho) \in \mathfrak{P}^\epsilon}\|\alpha\|\sup_{0 \le \tau \le 1}\Lambda(\rho,\tau)\exp\left(\int_{0}^1\Lambda(\rho,\tau')\rmd\tau'\right)\max_{1\le i\le n}(1+\|x_i\|)\\
&\le (1+\epsilon)\max_{1\le i\le n}(1+\|x_i\|)|\tau-\tau'|,
\end{align}
where we use
\begin{equation}
    \sup_{0 \le \tau \le 1}\|Z_\rho(\tau,x^i)\| \le \exp\left(\int_{0}^1\Lambda(\rho,\tau')\rmd\tau'\right)(1+ \|x_i\|),
\end{equation}
which can be proved using the same argument in the proof of claim (a).
Hence $R_\tau^\eps$ is continuous.

Next we fix $\tau \in (0,1)$ and compute
\begin{align}
    \overline{\lim\limits_{h \rightarrow 0}}\frac{R_{\tau+h}^\eps -R_\tau^\eps}{h} &\le \frac{1}{n}\EE   \overline{\lim\limits_{h \rightarrow 0}} \frac{1}{h}\left[\sup_{(\alpha,\rho) \in \mathfrak{P}^\epsilon}\sum_{i=1}^n\xi_i \alpha\transpose Z_\rho(\tau+h,x^i) - \sup_{(\alpha,\rho) \in \mathfrak{P}^\epsilon}\sum_{i=1}^n\xi_i \alpha\transpose Z_\rho(\tau,x^i)\right]\notag \\
    &\le \frac{1}{n}\EE   \overline{\lim\limits_{h \rightarrow 0}} \frac{1}{h}\sup_{(\alpha,\rho) \in \mathfrak{P}^\epsilon}\left[\sum_{i=1}^n\xi_i \alpha\transpose Z_\rho(\tau+h,x^i) - \sum_{i=1}^n\xi_i \alpha\transpose Z_\rho(\tau,x^i)\right]\notag \\
    &\le \frac{1}{n}\EE \sup_{(\alpha,\rho) \in \mathfrak{P}^\epsilon}\sum_{i=1}^n \xi_i \alpha\transpose\frac{\rmd}{\rmd \tau}Z_{\rho}(\tau,x^i) \notag \\
    &= \frac{1}{n}\EE \sup_{(\alpha,\rho) \in \mathfrak{P}^\epsilon}\sum_{i=1}^n\xi_i \int_{\mathbb{S}^{D-1}}\sigma(\omega \cdot Z_\rho(\tau,x^i))\rmd \alpha\transpose \rho_\tau(\omega) \notag \\
    &\le \frac{1}{n} \sup_{(\alpha,\rho)\in \mathfrak{P}^\epsilon}\|\alpha\|\Lambda(\rho,\tau)\exp\left(\int_{0}^{\tau}\Lambda(\rho,\tau')\rmd \tau'\right) \\
    & \qquad \times \EE\left[ \sup_{\|\omega\| \le 1,\rho \in \Psi }\Bigl|\sum_{i=1}^n\xi_i\sigma\Bigl(\frac{\omega \cdot Z_{\rho}(\tau,x^i)
    }{\exp(\int_{0}^\tau\Lambda(\rho, \tau')\rmd \tau')}\Bigr)\Bigr|\right] \notag \\
    &\le \frac{2(1+\epsilon)}{n}\EE \sup_{\|\omega\| \le 1,\rho \in \Psi }\sum_{i=1}^n\xi_i\Bigl(\frac{\omega\transpose Z_{\rho}(\tau,x^i)
    }{\exp(\int_{0}^\tau\Lambda(\rho,\tau')\rmd \tau')}\Bigr).
\end{align}
For any $\rho \in \Psi$, 
using a similar argument to the proof of equation \eqref{control of derivative}, we can find another $\overline{\rho} \in \Psi$ such that  $Z_{\overline{\rho}}(\tau,x) = Z_\rho(\tau,x)$ and 
\begin{align}
    \exp\left(\int_{0}^\tau\Lambda(\overline{\rho},\tau')\rmd \tau'\right) &\le (1+\epsilon) \exp\left(\int_{0}^\tau\Lambda(\rho, \tau')\rmd \tau'\right), \notag \\
    \sup_{0\le u \le \tau}\Lambda(\overline{\rho},u)\exp\left(\int_{0}^{u}\Lambda(\overline{\rho},\tau')\rmd \tau'\right) &\le (1+\epsilon)\exp\left(\int_{0}^\tau\Lambda(\rho, \tau')\rmd \tau'\right).
\end{align}
We  can then find another $\hat{\rho} \in \Psi$ such that  $\hat{\rho}_{\tau'} = \overline{\rho}_{\tau'}$ for all $0 \le \tau' \le \tau$ and
\begin{align}
      \exp\left(\int_{0}^1\Lambda(\hat{\rho},\tau')\rmd \tau'\right) &\le (1+\epsilon)^2 \exp\left(\int_{0}^\tau\Lambda(\rho,\tau')\rmd \tau'\right),\notag \\
    \sup_{0\le u \le 1}\Lambda(\hat{\rho},u)\exp\left(\int_{0}^{u}\Lambda(\hat{\rho},\tau')\rmd \tau'\right) &\le (1+\epsilon)^2\exp\left(\int_{0}^\tau\Lambda(\rho,\tau')\rmd \tau'\right).
\end{align}
Hence, $Z_{\hat{\rho}}(\tau,x) = Z_\rho(\tau,x)$, and for any $\omega$ with $\|\omega\| \leq 1$,
\begin{equation}
    \left(\frac{\omega}{(1+\epsilon)^2\exp(\int_{0}^\tau\Lambda(\rho,\tau')\rmd \tau')}, \hat{\rho}\right) \in \mathfrak{P}^\epsilon.
\end{equation}
Therefore,
\begin{equation}
     \overline{\lim\limits_{h \rightarrow 0}}\frac{R_{\tau+h}^\eps -R_\tau^\eps}{h} \le \frac{2(1+\epsilon)^3}{n} \EE\left[\sup_{(\alpha,\rho) \in \mathfrak{P}^\epsilon}\sum_{i=1}^n\xi_i \alpha\transpose Z_\rho(\tau,x^i)\right] = 2(1+\epsilon)^3R_\tau^\eps,
\end{equation}
which means
\begin{equation}
     \overline{\lim\limits_{h \rightarrow 0}}\frac{\exp\{-2(1+\epsilon)^3(\tau+h)\}R_{\tau+h}^\eps - \exp\{-2(1+\epsilon)^3\tau\}R_\tau^\eps}{h} \le 0. 
\end{equation}
For any $\epsilon' > 0$, let $P_\tau = \exp\{-2(1+\epsilon)^3 \tau\} R_\tau^\eps - \epsilon' \tau$, then $P_\tau$ is continuous in $\tau$ and satisfies
\begin{equation}
      \overline{\lim\limits_{h \rightarrow 0}}\frac{P_{\tau+h} -P_\tau}{h} \le -\epsilon' < 0, 
\end{equation}
which means that $P_\tau$ is decreasing. Therefore,
\begin{equation}
    \exp\{-2(1+\epsilon)^3\} R_1^\eps -\epsilon'\le R_0^\eps, 
\end{equation}
or $R_1^\eps \le e^{2(1+\eps)^3} R_0^\eps$ by letting $\eps' \to 0$. We can conclude our proof by computing
\begin{align}
    R_0^\eps &= \frac{1}{n}\EE \sup_{\|\alpha\| \le 1+\epsilon}\sum_{i=1}^n\xi_i((x^i)\transpose,1,0_{D-d-1})\alpha = \frac{1+\epsilon}{n}\EE \Big\|\sum_{i=1}^n\xi_i((x^i)\transpose,1,0_{D-d-1})\Big\|\\
    &\le \frac{1+\epsilon}{n}\sqrt{\sum_{i=1}^n(\|x^i\|^2 + 1)},   
\end{align}
and letting $\eps \to 0$ in $\mathrm{Rad}_n({\Phi},\mathcal{X})\le R_1^\eps \leq e^{2(1+\eps)^3}R_0^\eps$.

Finally, claims (a) and (c) together imply that ${\Phi}$ satisfies the Assumption~\ref{function_class_assumption} (b)--(c) with $A_1 = A_2 = 1$ and $A_3 = e^2$ and claim (b) implies ${\Phi}$ satisfies Assumption~\ref{function_class_assumption} (a) when $D\ge d+2$.

\end{proof}

\section{Application to the McKean-Vlasov SDE}\label{sec:mvsde}

This section presents an application of our proposed IPM to McKean-Vlasov Stochastic Differential Equation (SDE). Throughout this section, we fix a complete filtered probability space $(\Omega, \mc{F}, \mathbb{F}= \{\mc{F}_t\}_{0 \leq t \leq T}, \PP)$, supporting $n+1$ independent $m$-dimensional Brownian motions $\{W^i\}_{i=1}^n$ and $W$, as well as i.i.d. $\mc{F}_0$-measurable $\RR^d$-valued random variables $\{\eta^i\}_{i=1}^n$ with law $\eta$ and $\EE\|\eta\|^2 < + \infty$. We are interested in the rate of convergence as $n \to \infty$ of an $n$-interacting particle system satisfying:
\begin{align}\label{n-body SDE}
    \rmd X_t^{n,i} &= B(t,X_t^{n,i},\bar\mu_t^n)\rmd t + \Sigma(t,X_t^{n, i},\bar\mu_t^n)\rmd W_t^i, \quad X_0^{n,i} = \eta^i, \quad i \in \mc{I} := \{1,\dots, n\}, \notag \\
    \bar\mu_t^n &:= \frac{1}{n}\sum_{i=1}^n \delta_{X_t^{n,i}}.
\end{align}
More precisely, let $X_t$ solve the McKean-Vlasov stochastic differential equation:
\begin{align}\label{MV-SDE}
    \rmd X_t &= B(t,X_t,\mu_t)\rmd t + \Sigma(t,X_t,\mu_t)\rmd W_t, \quad X_0 = \eta, \notag \\
    \mu_t &:= \MCL(X_t),
\end{align}
where $\MCL(X_t)$ denotes the law of $X_t$, we are interested in quantifying $\EE \left[\sup_{0\le t \le T} \rmD_{{\Phi}}^2(\mu_t,\bar\mu_t^n)\right]$. To this end, we consider the following assumption. 

\begin{assu}\label{assu_MV}

The functions $B: [0,T]\times \bR^d \times \mathcal{P}^2(\bR^d) \rightarrow \bR^d$ and $\Sigma: [0,T]\times\bR^d \times \mathcal{P}^2(\bR^d) \rightarrow \bR^{d\times m}$ are Lipschitz in $(x, \mu)$ in the sense:
\begin{equation}\label{assump:MVSDE}
    \|B(t,x,\mu) - B(t,x',\mu')\|^2 + \|\Sigma(t,x,\mu) - \Sigma(t,x,\mu')\|_F^2 \le K^2(\|x-x'\|^2 + \rmD_{\Phi}^2(\mu,\mu')),
\end{equation}
for all $t \in [0,T]$, $x, x' \in \RR^d$ and $\mu, \mu' \in \mathcal{P}^2(\RR^d)$. Here $D_\Phi$ denotes a integral probability metric with the test function class $\Phi$ satisfying Assumption~\ref{function_class_assumption},
$\|\cdot\|_F$ denotes the Frobenius norm on $\bR^{d \times m}$, and $K$ is a positive constant.

Also, assume that
\begin{equation}
    \sup_{t \in [0,T]}\|B(t,0,\delta_0)\| + \|\Sigma(t,0,\delta_0)\|_F \le K.
\end{equation}
\end{assu}

\begin{rem}
Inequality \eqref{assump:MVSDE} is satisfied, for instance, 
when $B$ and $\Sigma$ are of the form
\begin{equation}
    (B, \Sigma) = (B, \Sigma)\left(t, x, \int_{\RR^d} f^1(x,y) \ud \mu(y),\int_{\RR^d} f^2(x,y) \ud \mu(y),\dots,\int_{\RR^d} f^k(x,y) \ud \mu(y) \right),
\end{equation}
and are Lipschitz in their second and third arguments with
$f^i(x,\cdot)$ in the class of test functions $\Phi$ satisfying Assumption~\ref{function_class_assumption}, for any $1\le i \le k$ and $x \in \bR^d$.
\end{rem}

\begin{thm}\label{thm:MV-Expectation}
Under Assumptions \ref{function_class_assumption} and \ref{assu_MV}, and the assumption that $\EE|\eta|^2\le K^2$, we have:
\begin{enumerate}
    \item[(a)] There exist unique adapted $L^2$-solutions for the $n$-body SDE \eqref{n-body SDE} and the McKean-Vlasov SDE \eqref{MV-SDE}.
    \item[(b)] There exists a constant $C >0$, depending only on $K$ and $T$, such that
    \begin{equation}
        \EE\sup_{0\le t \le T}\|X_t^{n,1}\|^2 \le C, \quad \EE\sup_{0\le t \le T}\|X_t\|^2 \le C.
    \end{equation}
    \item[(c)] There exists a constant $C>0$, depending only on $K$, $T$, $A_1$, $A_2$ and $A_3$, such that
    \begin{equation}
        \EE \sup_{0\le t \le T}\rmD^2_{\Phi}(\mu_t,\bar{\mu}_t^n) \le C\frac{\log n}{n}.
    \end{equation}
\end{enumerate}
\end{thm}
\begin{rem}
As will seen in the proof, Theorem~\ref{thm:MV-Expectation} (c) relies on the estimates in Theorem~\ref{thm: process}, about which we have mentioned in Remark~\ref{rem:log} that the logarithm term can be removed from $\phi(n)$ for all the examples of $\Phi$ mentioned in Section~\ref{sec:testfunction}. Consequently, claim (c) in Theorem~\ref{thm:MV-Expectation} can be further improved to be bounded by $C/n$ when using test functions in Section~\ref{sec:testfunction}.

\end{rem}
\begin{proof}
Throughout the proof, we will use $C$ as a generic positive constant depending only on $K$, $T$, $A_1$, $A_2$ and $A_3$, which may vary from line to line.

By the relation between $D_\Phi$ and $\mathcal{W}_2$ stated in Theorem \ref{thm: iid}~(a), Claim (a) follows from Lemma 3.2 and Theorem 3.3 in \cite{lacker2018mean}.

For claim (b), define $\bx = [x^1,\cdots, x^n]\transpose$, $L_n(\bx) = \frac{1}{n}\sum_{i=1}^n \delta_{x^i}$ and
\begin{align}
    \bm{B}(t,\bx) &= [B(t,x^1,L_n(\bx)),\dots,B(t,x^n,L_n(\bx))]\transpose, \\
    \bm{\Sigma}(\bx) &= \left[
    \begin{matrix}
    \Sigma(t,x^1,L_n(\bx))& & & \\
    &\Sigma(t,x^2,L_n(\bx))&&\\
    &&\ddots&\\
    &&&\Sigma(t,x^n,L_n(\bx))\end{matrix}\right],
\end{align}
where $\bm{\Sigma}$ has zero entries except for the $n$ blocks of size $d\times m$ on the main diagonal. 
Then we can rewrite the $n$-body SDE \eqref{n-body SDE} as 
\begin{equation}\label{Vector n-body SDE}
    \rmd \bX_t^n = \bm{B}(t,\bX_t^n)\rmd t + \bm{\Sigma}(t,\bX_t^n) \rmd \bm{W}_t,
\end{equation}
where $\bX_t^n = [X_t^{n,1},\cdots,X_t^{n,n}]\transpose$ and $\bm{W}_t = [W_t^1,\cdots,W_t^n]\transpose$.
Following Lemma 3.2 in \cite{lacker2018mean}, we obtain that $\bm{B}$ and $\bm{\Sigma}$ are $2L$-Lipschitz. Standard SDE estimates ({\it cf.} \cite[Theorem~3.2.2]{zhang2017backward}) give
\begin{equation}
    \sup_{0 \le t \le T} \EE \sum_{i=1}^n \|X_t^{n,i}\|^2 \le Cn.
\end{equation}
Notice that $X_t^{n,1},\dots,X_t^{n,n}$ are symmetric, one has
\begin{equation}
   \sup_{0\le t \le T} \EE\|X_t^{n,1}\|^2 \le C.
\end{equation}
Then, using the Burkholder-Davis-Gundy inequality ({\it cf.}~\cite[Theorem~2.4.1]{zhang2017backward}), we have
\begin{align}
    \EE\sup_{0\le t \le T}\|X_t^{n,1}\|^2 &\le C\left[\EE\|\eta\|^2 + \EE \int_{0}^{T}|B(t,X_t^{n,1},\bar{\mu}_t^n)|^2\rmd t + \EE \sup_{0\le u \le T}\|\int_{0}^{T}\Sigma^i(t,X_t^{n,1},\bar{\mu}_t^n)\rmd W_t\|^2\right] \\
    &\le C\left[\EE\|\eta\|^2 + \EE \int_{0}^{T}|B(t,X_t^{n,1},\bar{\mu}_t^n)|^2\rmd t + \EE\int_{0}^{T}\|\Sigma(t,X_t^{n,1},\bar{\mu}_t^n)\|_F^2\rmd t\right] \\
    &\le C[1 + \sup_{0\le t \le T}\EE\|X_t^{n,1}\|^2 + \sup_{0\le t \le T}\rmD^2_\Phi(\bar{\mu}_t^n,\delta_0)] \le C[1 + \sup_{0\le t \le T}\EE\|X_t^{n,1}\|^2] \le C.
\end{align}
By \cite[Theorem 3.3]{lacker2018mean}, we know that, as $n\rightarrow\infty$,
\begin{equation}
    X^{n,1} \Rightarrow X, \text{ in distribution in } C([0,T];\bR^d).
\end{equation}
We then obtain the second inequality in claim (b) through the Fatou's Lemma.

For claim (c), let 
\begin{equation}
    \rmd Y_t^{n,i} = B(t,Y_t^{n,i},\mu_t)\rmd t + \Sigma(t,Y_t^{n,i},\mu_t)\rmd W_t^i,\quad Y_0^{n,i} = \eta^i.
\end{equation}
Then, $\{Y_t^{n,i}\}_{i=1}^n$ are $n$ i.i.d. copies of $X_t$. Following \cite[Theorem 3.2.4]{zhang2017backward}, we obtain, $\forall t \in [0,T]$,
\begin{multline}
      \EE\Big[\sup_{0\le s \le t}\|X_s^{n,i} - Y_s^{n,i}\|^2\Big] \\
      \le C\EE\int_{0}^{t}[\|B(s,X_s^{n,i},\bar{\mu}_s^n) - B(s,Y_s^{n,i},\mu_s)\|^2 + \|\Sigma(s,X_s^{n,i},\bar{\mu}_s^n) - \Sigma(s,Y_s^{n,i},\mu_s)\|_F^2] \rmd s.
\end{multline}
Using the Lipschitz condition \eqref{assump:MVSDE} in Assumption \ref{assu_MV}, we deduce
\begin{equation}
     \EE\Big[\sup_{0\le s \le t}\|X_s^{n,i} - Y_s^{n,i}\|^2\Big] \le C\EE\int_{0}^{t}\Big[\sup_{0\le u\le s}\|X_u^{n,i} - Y_u^{n,i}\|^2 + \rmD_{\Phi}^2(\bar{\mu}_s^n,\mu_s)\Big] \rmd s. 
\end{equation}
Then Gronwall's inequality gives
\begin{equation}
    \EE\Big[\sup_{0\le s \le t}\|X_s^{n,i} - Y_s^i\|^2\Big] \le C\EE\Big[\int_{0}^{t}\rmD_{\Phi}^2(\bar{\mu}_s^n,\mu_s) \rmd s \Big].
\end{equation}
Let $\hat \mu_t^n$ be the empirical measure of $\{Y_t^{n,i}\}_{i=1}^n$, {\it i.e.},
\begin{equation}
    \hat{\mu}_t^n := \frac{1}{n}\sum_{i=1}^n \delta_{Y_t^{n,i}}.
\end{equation}
With $\sup_{0 \le t\le T}\rmD_{\Phi}(\mu_t,\delta_0) \le A_1\sup_{0\le t \le T}\EE\|X_t\|\le C$, we obtain
\begin{equation}
    \sup_{0\le t \le T}[\|B(t,0,\mu_t)\| + \|\Sigma(t,0,\mu_t)\|_F] \le C.
\end{equation}
Viewing $\mu_t$ as a given function of $t$,  the McKean-Vlasov SDE \eqref{MV-SDE} satisfies the conditions in Theorem \ref{thm: prior_estimation}. Thus, combining results in Theorems~\ref{thm: process} and \ref{thm: prior_estimation}, we have
\begin{equation}
    \EE \sup_{0\le t \le T}\rmD_{\Phi}^2(\mu_t,\hat{\mu}_t^n) \le C\frac{\log n}{n}. 
\end{equation}
By the definition of $\rmD_{\Phi}$, one has
\begin{equation}
    \rmD_{\Phi}^2(\bar{\mu}_t^n,\hat{\mu}_t^n) \le \Big[\frac{A_1}{n}\sum_{i=1}^n\|X_t^{n,i}- Y_t^{n,i}\|\Big]^2 \le \frac{A_1^2}{n}\sum_{i=1}^n \|X_t^{n,i}- Y_t^{n,i}\|^2. 
\end{equation}
Therefore
\begin{multline}
      \EE\Big[\sup_{0\le s \le t} \rmD^2_{\Phi}(\mu_s,\bar{\mu}_s^n)\Big] \le 2 \EE\Big[\sup_{0 \le s \le t}\rmD^2_{\Phi}(\bar{\mu}_s^n,\hat{\mu}_s^n) + \sup_{0\le s \le t}\rmD^2_{\Phi}(\mu_s,\hat{\mu}_s^n)\Big] \\
      \le C\EE\Big[\int_{0}^{t}\rmD_{\Phi}^2(\mu_s,\bar{\mu}_s^n) \rmd s\Big] + C\frac{\log n}{n}.
\end{multline}
With Gronwall's inequality, we can obtain the desired result.
\end{proof}

We can furthermore establish a concentration inequality for $\sup_{0\le t \le T}\rmD_{\Phi}(\mu_t,\bar{\mu}_t^n)$.
\begin{thm}
Under Assumptions \ref{function_class_assumption} and \ref{assu_MV}, and assume that
\begin{equation}
   \mathcal{W}_1^2(\mu_0,\tilde{\mu}) \le 2K^2\mathcal{H}(\tilde{\mu}|\mu_0) \quad \forall \tilde{\mu} \ll \mu_0, \text{ and } \EE\|\eta\|
    ^2 \le K^2,
\end{equation}
and
\begin{equation}
    \sup_{0\le t \le T, x \in \bR^d}\|\Sigma(t,x)\|_F \le K.
\end{equation}
Then, there exists a constant $C > 0$, depending only on $T$, $K$ and $A_1$, such that
\begin{equation}
    \mathbb{P}(\sup_{0\le t \le T} \rmD_{\Phi}(\mu_t,\bar{\mu}_t^n) - \EE\sup_{0\le t \le T} \rmD_{\Phi}(\mu_t,\bar{\mu}_t^n) \ge a) \le \exp(-\frac{na^2}{C}).
\end{equation}
Combining with Theorem \ref{thm:MV-Expectation}, we obtain that for any $\delta \in (0,1)$, with probability at least $1-\delta$
\begin{equation}
    \sup_{0\le t \le T}\rmD_{\Phi}(\mu_t,\bar{\mu}_t^n) \le C(\sqrt{\log n} + \sqrt{-\log\delta})n^{-\frac{1}{2}},
\end{equation}
where $C$ may depend on $K$, $T$, $A_1$, $A_2$ and $A_3$. 
\end{thm}
\begin{proof}
Throughout this proof, we will still use $C$ as a positive constant depending only on some constants clearly mentioned in the above theorem, which may vary from line to line.

Recall that the $n$-particle system can be rewritten as in \eqref{Vector n-body SDE}, and the results from \cite[Theorem 5.5]{delarue2020master}: there exists a constant $C>0$, depending only on $K$ and $T$, such that 
\begin{equation}\label{concentration-inequality}
    \mathbb{P}( F(\bX^n) - \EE F(\bX^n) \ge a) \le \exp(-\frac{a^2}{n M^2C}),
\end{equation}
for any function $F:C([0,T];\bR^{d\times n}) \rightarrow \bR$ being $M$-Lipschitz in the sense that 
\begin{equation}
    |F(\bm{x}) - F(\bm{y})| \le M \sum_{i=1}^n\sup_{0\le t \le T}\|x^i_t - y^i_t\|,
\end{equation}
for any $\bm{x} := (x^1,\dots,x^n), \bm{y} := (y^1,\dots,y^n)$ with $x^i,y^i \in C([0,T];\bR^{d})$.

Now, for any $\bm{x} = (x^1,\dots,x^n)$ with $x^i \in C([0,T];\bR^{d})$, we define 
\begin{equation}
    G(\bm{x}) = \sup_{0\le t \le T}\sup_{f \in \Phi}\Bigl|\frac{1}{n}\sum_{i=1}^{n} f(x^i_t) - \EE f(X_t)\Bigr|,
\end{equation}
then $G(\bX^n) = \sup_{0\le t \le T}\rmD_{\Phi}(\mu_t,\bar{\mu}_t^n)$, and
\begin{align}
    |G(\bm{x}) - G(\bm{y})| &\le \frac{1}{n}\sup_{0\le t \le T}\sup_{f \in \Phi}\Bigl|\sum_{i=1}^n f(x^i_t) - \sum_{i=1}^n f(y^i_t)\Bigr| \notag \\
    &\le  \frac{A_1}{n}\sum_{i=1}^n\sup_{0\le t \le T}\|x^i_t- y^i_t\| .
\end{align}
Then, our conclusion follows from the last equation and equation \eqref{concentration-inequality}.
\end{proof}

\section{Application to Mean-Field Games}

In this section, we shall show that, for a homogeneous $n$-player game, the strategy derived by its mean-field counterpart produces an $\varepsilon$-Nash equilibrium, where $\varepsilon$ is free of the dimension of the state processes.

Following the setup in \cite{carmona2013probabilistic}, we consider a homogeneous $n$-player stochastic differential game
\begin{equation}\label{def:Xt}
    \ud X_t^i = b(t, X_t^i, \bar \nu_t^n, \alpha_t^i) \ud t + \sigma(t, X_t^i, \bar \nu_t^n, \alpha^i) \ud W_t^i, \quad 0 \leq t \leq T, \quad i \in \mc{I} \equiv \{1, \ldots, n\},
\end{equation}
where each player $i$ controls her private state $X_t^i \in \RR^d$ through an $\RR^k \supseteq A$-valued action $\alpha_t^i$, $W_t^i = (W_t^i)_{0 \leq t \leq T}$ are $m$-dimensional independent Brownian motions, $b$ and $\sigma$ are deterministic measurable functions, $(b, \sigma)$: $[0,T] \times \bR^d \times \mathcal{P}(\bR^d) \times A \to (\bR^d, \bR^{d\times m})$, and $\bar\nu_t^n$ is the empirical measure of $(X_t^1, \ldots, X_t^n)$ defined by 
\begin{equation}
    \bar\nu_t^n(\mathrm{d} x) := \frac{1}{n} \sum_{i=1}^n \delta_{X_t^i} (\mathrm{d} x).
\end{equation}
Each player aims to minimize the expected cost over the period $[0,T]$ by taking her action $\alpha^i \in \mbA$:
\begin{equation}\label{def:J}
    J^{i}(\balpha) := \EE\left[\int_0^T f(t, X_t^i, \bar \nu_t^n, \alpha^i_t) \ud t + g(X_T^i, \bar \nu_T^n)\right],
\end{equation}
where $\mbA$ denotes the set of all admissible strategies: 
\begin{equation}
    \mbA = \left\{A\text{-valued progressively measurable processes } (\alpha_t)_{0\leq t \leq T}: \EE\left[\int_0^T \abs{\alpha_t}^2 \ud t\right] < \infty\right\},
\end{equation}
and $f$ and $g$ are deterministic measurable functions, $f: [0,T] \times \bR^d\times \MCP(\bR^d)\times A \to \bR$,  $g: \bR^d \times \MCP(\bR^d) \to \RR$. Since the players interact through their empirical measure $\bar \nu_t^n$, which depends on all players' stragety $\balpha = (\alpha^1, \ldots, \alpha^n) \in \mbA^n$, so does the cost functional for player  $i$, $J^i(\balpha)$. Here $\mbA^n$ is the product space of $n$ copies of $\mbA$.

To solve such games, we are interested in the concept of Nash equilibrium. That is a tuple $\bm \alpha^\ast = (\alpha^{1,\ast}, \ldots, \alpha^{n, \ast}) \in \mathbb{A}^n$ such that 
\begin{equation}
\forall i \in \mc{I}, \text{ and } \alpha^i \in \mathbb{A}, \quad J^i(\bm \alpha^\ast)  \leq J^i(\alpha^{1, \ast}, \ldots, \alpha^{i-1, \ast}, \alpha^i, \alpha^{i+1, \ast}, \ldots, \alpha^{n, \ast}).
\end{equation}

For homogeneous games with large $n$, if the system lacks tractability and needs to rely on numerical methods for Nash equilibrium, the conventional algorithms soon lose their efficiency, and one may resort to recently developed machine learning tools \cite{Hu2:19,HaHu:19,han2020convergence}. On the other hand, one could utilize its limiting mean-field strategy to approximate the Nash equilibrium. More precisely, one can first obtain the optimal control $\alpha$ from the mean field games using the following steps: 
\begin{enumerate}[(i)]
    \item\label{eq:step1} Fixed a deterministic measure $\mu_t \in \MCP(\RR^d)$, $\forall t \in [0, T]$;
    \item\label{eq:step2} Solve the standard stochastic control problem:
        \begin{align}
            &\inf_{\alpha \in \mbA} \EE\left[\int_0^T f(t, X_t, \mu_t, \alpha_t)\ud t + g(X_T, \mu_T) \right] \\
            &\text{subject to: }\ud X_t = b(t, X_t, \mu_t, \alpha_t) \ud t + \sigma(t, X_t, \mu_t, \alpha_t) \ud W_t, \quad X_0 = x_0;
        \end{align}
    \item\label{eq:step3} Determine the flow of measures $\mu_t$ such that $\forall t \in [0,T]$, $\MCL({X_t^{\ast, \mu_t}}) = \mu_t$, where $X_t^{\ast, \mu_t}$ denotes the state process associated to the optimal control given $\mu_t$ in step~\eqref{eq:step2}.
\end{enumerate}
Then one can construct an $\varepsilon$-Nash equilibrium from it if the optimal control $\alpha$ given by the fixed-point argument (step~\eqref{eq:step3}) is in a feedback form.
We will make this statement rigorous in Theorem~\ref{thm:MFG}.

Throughout this section, the following assumptions are in force.
\begin{assu}\label{assu_MFG}
\begin{enumerate}[(a)]
    \item\label{assu_b} The drift $b$ is an affine function of $\alpha$ and $x$: 
    \begin{equation}
        b(t, x, \mu, \alpha) = b_0(t, \mu) + b_1(t)x + b_2(t)\alpha,
    \end{equation}
    where $b_0 \in \RR^d$, $b_1 \in \RR^{d\times d}$, $b_2 \in \RR^{d \times k} $ are measurable functions and bounded by $K$. Moreover, for any $\mu, \mu' \in \MCP^2(\RR^d)$: $\abs{b_0(t, \mu') - b_0(t, \mu)} \leq K \dF(\mu, \mu')$. The volatility $\sigma(t, x, \mu, \alpha) \in \RR^{d \times m}$ is a constant matrix. 
    
    \item\label{assu_fg} There exist two constants $\lambda$ and $K$, such that for any $(t, \mu) \in [0,T] \times \MCP^2(\RR^d)$, the function $f(t, \cdot, \mu, \cdot) \in \RR$ is once continuously differentiable with Lipschitz-continuous derivatives, with the Lipschitz constants being bounded by $K$. Moreover, it satisfies the convexity assumption:
    \begin{equation}
        f(t, x', \mu, \alpha') - f(t, x, \mu, \alpha) - \average{(x' - x, \alpha' - \alpha), \partial_{(x, \alpha)}f(t, x, \mu, \alpha)} \geq \lambda \abs{\alpha' - \alpha}^2.
    \end{equation}
    The functions $f, \partial_x f$, and $\partial_\alpha f$ are locally bounded. 
        The functions $f(\cdot, 0, \delta_0, 0)$, $\partial_x f(\cdot, 0, \delta_0, 0)$ and $\partial_\alpha f(\cdot, 0, \delta_0, 0)$ are bounded by $K$, and for all $t \in [0,T]$, $x, x' \in \RR^d$, $\alpha, \alpha' \in \RR^k$ and $\mu, \mu' \in \MCP^2(\RR^d)$, it holds:
    \begin{align}
        \abs{(f,g)(t, x', \mu', \alpha') - (f,g)(t, x, \mu, \alpha)} \leq &K\left[1 + \abs{(x', \alpha')} +\abs{(x,\alpha)}+  M_2(\mu) + M_2(\mu') \right] \\
        &\times \left[\abs{(x', \alpha') - (x, \alpha)} + \dF(\mu', \mu)\right] \label{eq:fglip}.
    \end{align}
    
    \item The function $g(\cdot, \cdot)$ is locally bounded, and for any $\mu \in \MCP^2(\RR^d)$, the function $g(\cdot, \mu)$ is once continuously differentiable, convex, and has a $K$-Lipschitz-continuous first order derivative. 

    \item For all $(t, x, \mu) \in [0,T] \times \RR^d \times \MCP^2(\RR^d)$, $\abs{\partial_x f(t,x,\mu,0)} \leq K$.
    \item For all $(t,x) \in [0,T] \times \RR^d$, $\average{x, \partial_x f(t, 0, \delta_x, 0)} \geq -K(1+\abs{x})$, $\average{x, \partial_x g(0, \delta_x)}\geq -K(1+\abs{x})$.
\end{enumerate}
\end{assu}

In the sequel, a constant $C$ will frequently appear in the theorems and proofs. It may depend on the bounds that appear in the above assumption 
($\lambda$, $K$, $b_0$, $b_1$, $A_1$, $A_2$, $A_3$, {\it etc.}) and possibly vary from line to line. But it will be  independent of the dimension $d$  of the state process $X_t^i$ and the number of players $n$ in the game.

\begin{rem}\label{rem:alphahat}
Let $H$ be the Hamiltonian associated to the problem, with uncontrolled volatility, it reads
\begin{equation}\label{def:H}
    H(t,x,\mu, p, \alpha) = \average{b(t, x, \mu, \alpha), p} + f(t, x, \mu, \alpha).
\end{equation}
Items \eqref{assu_b}--\eqref{assu_fg} in Assumption~\ref{assu_MFG} ensure the uniqueness of minimizer $\hat\alpha$ of $H$, the measurablity, local boundedness,  Lipschitz-continuity of $\hat \alpha(t,x,\mu,y)$ in $(x,y)$ uniformly in $(t, \mu) \in [0,T] \times \mc{P}^2(\RR^d)$. Moreover, the Lipschitz constant is free of $d$. A repeatedly used estimate is $\abs{\hat\alpha(t, x, \mu, y)} \leq \lambda^{-1}(\abs{\partial_\alpha f(t, x, \mu, 0)} + \abs{b_2(t)}\abs{y})$.
For detailed proof, see \cite[Lemma 1]{carmona2013probabilistic}.
\end{rem}

The probabilistic approach of \eqref{eq:step1}--\eqref{eq:step3} results in solving the following McKean-Vlasov forward backward stochastic differential equations (FBSDEs):
\begin{align}\label{def:FBSDE}
  \begin{aligned}
    &\ud X_t = b(t, X_t, \MCL({X_t}), \hat\alpha(t, X_t, \MCL({X_t}), Y_t)) \ud t + \sigma \ud W_t, \\
    &\ud Y_t = -\partial_x H(t, X_t, \MCL({X_t}), Y_t, \hat\alpha(t, X_t, \MCL({X_t}), Y_t)) \ud t + Z_t \ud W_t,
\end{aligned}  
\end{align}
with the initial condition $X_0 = x_0 \in \RR^d$ and the terminal condition $Y_T = \partial_x g(X_T, \MCL({X_T}))$. More precisely, the following result holds.
\begin{thm}\label{thm:FBSDE}
Under Assumption~\ref{assu_MFG}, the FBSDE system \eqref{def:FBSDE} has a solution $(X_t, Y_t, Z_t)$, and there exists the FBSDE value function $u: [0,T] \times \RR^d \to \RR^d$ such that it has linear growth and Lipschitz in $x$:
\begin{equation}\label{eq:u}
    \abs{u(t, x)} \leq C(1+\abs{x}), \quad \abs{u(t, x) - u(t, x')} \leq C\abs{x - x'}, 
\end{equation}
for some constant $C \geq 0$, and such that $Y_t = u(t, X_t)$ $\PP$-a.s., $\forall t \in [0,T]$ and $x, x' \in \RR^d$. Moreover, for any $\ell \geq 1$, $\EE[\sup_{0 \leq t \leq T}\abs{X_t}^\ell] < \infty$, and the optimal cost $J$ of the limiting mean-field problem \eqref{eq:step1}--\eqref{eq:step3} is given by
\begin{equation}\label{def:mfgJ}
      J = \EE\left[g(X_T, \MCL({X_T})) + \int_0^T f(t, X_T, \MCL({X_t}), \hat\alpha(t, X_t, \MCL({X_t}), Y_t)) \ud t\right],
\end{equation}
where $\hat \alpha$ is the minimizer of $H$ defined in \eqref{def:H}.
\end{thm}
\begin{proof}
The existence of solution to \eqref{def:FBSDE} and related estimates follow from \cite[Theorem 2]{carmona2013probabilistic}, because the assumptions therein are satisfied using Theorem~\ref{thm: iid} (a) based on the Wasserstein metric and our proposed IPM. The statement on $J$ is a consequence of the stochastic maximum principle when the frozen flow of measures is $\MCL({X_t})$, for instance see \cite[Theorem 1]{carmona2013probabilistic}.
\end{proof}

\begin{thm}\label{thm:MFG}
Let $(X_t, Y_t, Z_t)$ be a solution of \eqref{def:FBSDE}, $u$ be the corresponding FBSDE value function, and $\mu_t = \MCL({X_t})$ be the marginal probability measure, then 
\begin{equation}\label{def_approxNash}
    \bar\alpha_t^{n,i} = \hat\alpha(t, X_t^i, \mu_t, u(t, X_t^i)), \quad i \in \mc{I},
\end{equation}
where $X_t^i$ follows \eqref{def:Xt} with strategy $\bar\alpha_t^{n,i}$:
\begin{equation}\label{eq:Xt}
    \ud X_t^i = b(t, X_t^i, \bar\nu_t^n, \hat\alpha(t, X_t^i, \mu_t, u(t, X_t^i)) \ud t + \sigma \ud W_t^i,  \quad \bar\nu_t^n = \frac{1}{n} \sum_{i=1}^n \delta_{X_t^i},
\end{equation}
is an $\varepsilon_n$-Nash equilibrium of the $n$-player problem \eqref{def:Xt}--\eqref{def:J} with $\eps_n = C/\sqrt n$, {\it i.e.}, for any progressively measurable strategy $\beta^i$ such that $\EE[\int_0^T |\beta_t^i|^2\ud t]< \infty$, we have
\begin{equation}\label{eq:approxNash}
 J^{n, i}(\bar\alpha^{n,i},\ldots, \bar\alpha^{n, i-1}, \beta^i, \bar\alpha^{n, i+1}, \ldots, \bar\alpha^{n,n}) \geq    
 J^{n, i}(\bar\alpha^{n, 1}, \ldots, \bar\alpha^{n,n}) - \varepsilon_n.
\end{equation}
\end{thm}

\begin{proof}
We first claim that the SDE \eqref{eq:Xt} is well defined, by the Lipschitz property and linear growth of $\hat\alpha$ in $(x,y)$ and $u$ in $x$ ({\it cf.} Remark~\ref{rem:alphahat} and \eqref{eq:u}). With Assumption~\ref{assu_MFG} \eqref{assu_b}, we also have
\begin{equation}\label{eq:Xalphabound}
    \sup_{n \geq 1}\max_{1 \leq i \leq n}\left[\EE[\sup_{0 \leq t \leq T}\abs{X_t^i}^2] + \EE\int_0^T \abs{\bar \alpha_t^{n,i}}^2 \ud t\right] \le C.
\end{equation}
The proof of \eqref{eq:approxNash} consists of two steps, and by symmetry we only need to prove it for $i = 1$.

\medskip
\noindent\textbf{Step 1: MFG \emph{vs.} $N$-player game using $(\bar\alpha^{n,1}, \ldots, \bar\alpha^{n,n})$.}
To this end, we introduce $n$ independent copies of the mean-field states $X_t$ in \eqref{def:FBSDE}:
\begin{equation}
    \ud \barXti = b(t, \barXti, \mu_t, \hat\alpha(t, \barXti, \mu_t, u(t, \barXti) ) \ud t + \sigma \ud W_t^i, \quad t \in [0,T], \text{ and } i \in \mc{I}.
\end{equation}
Note that $\MCL({\barXti}) = \mu_t$, and we have similar estimates for $(\barXti, \hat\alpha_t^i)$ as in \eqref{eq:Xalphabound}. 
Let  $\bar\mu_t^n$ be the empirical measure of $\barXti$ and define,
\begin{equation}
    \hat \alpha_t^i = \hat\alpha(t, \barXti, \mu_t, u(t, \barXti)),
\end{equation}
we then compute, by the regularity of $b$, $u$ and $\hat\alpha$, that for $t \in [0,T]$:
\begin{align*}
\EE\left[\sup_{0 \leq s \leq t} \abs{X_s^i - \barXsi}^2 \right]  &\leq \EE\left[\int_0^t \abs{b(s, X_s^i, \bar\nu_s^n, \bar\alpha_s^{n,i}) - b(s, \barXsi, \mu_s, \hat\alpha_s^i)}^2  \ud s \right] \\
   &\leq C \EE\left[\int_0^t \abs{X_s^i - \barXsi}^2 +  \dF^2(\bar\nu_s^n, \mu_s) + \abs{\bar\alpha_s^{n,i} - \hat\alpha_s^i}^2 \ud s \right] \\
   & \leq C \EE\left[\int_0^t \abs{X_s^i - \barXsi}^2 +  \dF^2(\bar \nu_s^n, \mu_s) \ud s \right]. 
\end{align*}
Then Gronwall's inequality gives
\begin{equation}
   \EE\left[\sup_{0 \leq s \leq t} \abs{X_s^i - \barXsi}^2 \right]  \leq   C \EE\left[\int_0^t   \dF^2(\bar \nu_s^n, \mu_s) \ud s \right], \forall i \in \mc{I} \text{ and } t \in [0,T].
\end{equation}
A similar proof as in Theorem~\ref{thm: iid} gives
\begin{equation}\label{eq:dis_measure_barXt_Xt}
\EE   \dF^2(\mu_t, \bar\mu_t^n) \leq \frac{C}{n}, \quad \forall t \in [0,T],
\end{equation}
and by the definition of $\dF$ we have
\begin{equation}
   \dF^2(\bar\nu_t^n, \bar \mu_t^n)  \leq \frac{A_1^2}{n} \sum_{i=1}^n \abs{X_t^i - \barXti}^2, \quad \forall t \in [0,T].
\end{equation}
Thus one deduces
\begin{align}
\EE [  \dF^2(\mu_t, \bar\nu_t^n)] &\leq 2\EE [   \dF^2(\mu_t, \bar\mu_t^n)] + 2\EE [   \dF^2(\bar\mu_t^n, \bar\nu_t^n)]\\
& \leq \frac{C}{n} + C\EE\left[ \frac{1}{n} \sum_{i=1}^n \abs{X_t^i - \barXti}^2\right]  \leq  \frac{C}{n}  + C \EE\left[\int_0^t  \dF^2(\bar \nu_s^n, \mu_s) \ud s \right].
\end{align}
Applying Gronwall's inequality again yields
\begin{equation}\label{eq:dis_measure}
    \EE [  \dF^2(\mu_t, \bar\nu_t^n)] \leq \frac{C}{n}, \quad \forall t \in [0,T],
\end{equation}
and consequently 
\begin{equation}\label{eq:dis_state}
      \EE\left[\sup_{0 \leq s \leq t} \abs{X_s^i - \barXsi}^2 \right] \leq \frac{C}{n}, \quad \forall i \in \mc{I}, \quad \forall t \in [0,T].
\end{equation}
We are now ready to compare $J^{n,i}(\bar\alpha^{n,1}, \ldots, \bar\alpha^{n,n})$ with the mean-field problem value $J$ defined in \eqref{def:mfgJ}, which coincide with $\EE[\int_0^T f(t, \barXti, \mu_t, \hat\alpha_t^i) \ud t + g(\barXTi, \mu_T)]$, as $\barXti \stackrel{\mc{D}}{=} X_t$. By Assumption~\ref{assu_MFG} \eqref{assu_fg}, the Cauchy-Schwarz inequality, the boundedness of $(X_t^i, \barXti, \bar\alpha_t^{n,i}, \hat\alpha_t^i)$ in expectation ({\it cf.} Theorem~\ref{thm:FBSDE} and estimate \eqref{eq:Xalphabound}), the Lipschitz property of $\hat\alpha$ and $u$, we have
\begin{align}
    \abs{J - J^{n,i}(\bar\alpha^{n,1}, \ldots, \bar\alpha^{n,n})} &\leq C \EE\left[\abs{\barXTi-X_T^i}^2 + \dF^2(\mu_T, \bar \nu_T^n)\right]^{1/2} \\
    &\quad + C \left(\int_0^T \EE\left[\abs{\barXti-X_t^i}^2 + \dF^2(\mu_t, \bar \nu_t^n) \ud t \right]\right)^{1/2},
\end{align}
and then conclude
\begin{equation}\label{eq:mfgvsalpha}
    J^{n,i}(\bar\alpha^{n,1}, \ldots, \bar\alpha^{n,n}) = J + C/\sqrt n,
\end{equation} 
by the estimates \eqref{eq:dis_measure} and \eqref{eq:dis_state}.  This suggest that, in order to prove \eqref{eq:approxNash}, we only need to compare $J^{n,i}(\beta^1,  \bar\alpha^{n,2}, \ldots, \bar\alpha^{n,n})$ with $J$.

\medskip
\noindent\textbf{Step 2: MFG \emph{vs. }$N$-player game using $(\beta^1, \bar\alpha^{n,2}, \ldots, \bar\alpha^{n,n})$.}
Denote by $(U_t^1, \ldots, U_t^n)$ the solution to \eqref{def:Xt} using strategy $(\beta^1, \bar\alpha^{n,2}, \ldots, \bar\alpha^{n,n})$, and $\hat\nu_t^n$ the empirical measure of $(U_t^1, U_t^2, \ldots, U_t^n)$, and $\hat\nu_t^{n-1}$ the empirical measure of $(U_t^2, \ldots, U_t^n)$. By the boundedness of $b_0$, $b_1$ and $b_2$, the admissiblity of $(\beta^1, \bar\alpha^{n,2}, \ldots, \bar\alpha^{n,n})$, and Gronwall's inequality, we have the following estimates:
\begin{align}\label{eq:boundsofU}
    &\EE[\sup_{0 \leq t \leq T} |U_t^1|^2] \leq C(1 + \EE\int_0^T |\beta_t^1|^2\ud t ), \quad \EE[\sup_{0 \leq t \leq T} |U_t^j|^2] \leq C, \quad j \in \mc{I} \setminus \{1\}, \\
    &\frac{1}{n}\sum_{j=1}^n  \EE[\sup_{0 \leq t \leq T} |U_t^j|^2]   \leq C(1 + \frac{1}{n}\EE\int_0^T |\beta_t^1|^2\ud t ).\label{eq:boundsofavgU}
\end{align}

\noindent\textbf{Step 2.1: Controlling $\dF(\hat\nu_t^n, \mu_t)$.}
By triangle inequality of the IPM, one has
\begin{equation}\label{eq:distanceUXbar}
    \EE[\dF^2(\hat\nu_t^n, \mu_t)] \leq C \left\{ \EE[\dF^2(\hat\nu_t^n, \hat\nu_t^{n-1})] +  \EE[\dF^2(\hat\nu_t^{n-1}, \bar\mu_t^{n-1})] +  \EE[\dF^2(\bar\mu_t^{n-1}, \mu_t)]\right\},
\end{equation}
and the last term is $\mc{O}(1/n)$ by \eqref{eq:dis_measure_barXt_Xt}. For the first term, we have
\begin{equation}
    \EE[\dF^2(\hat\nu_t^n, \hat\nu_t^{n-1})] \leq \frac{C}{n(n-1)} \sum_{j=2}^n \EE[|U_t^1 - U_t^j|^2],
\end{equation}
and is $\mc{O}(1/n)$ using the estimate \eqref{eq:boundsofU}.
By definition, the second term in \eqref{eq:distanceUXbar} is bounded by
\begin{equation}
    \EE[\dF^2(\hat\nu_t^{n-1}, \bar\mu_t^{n-1})] \leq \frac{C}{n-1}\sum_{j=2}^n \EE[|U_t^j - \bar X_t^j|^2].
\end{equation}
For $2 \leq j \leq n$, we deduce that
\begin{equation}
    \EE[|U_t^j - \bar X_t^j|^2] \leq  2\EE[|U_t^j -  X_t^j|^2] +  2\EE[|X_t^j - \bar X_t^j|^2] \leq \frac{C}{n},
\end{equation}
by \eqref{eq:dis_state}, the estimates (following the derivation of (58) in \cite{carmona2013probabilistic}):
\begin{align}
  \sup_{0 \leq t \leq T} \EE[|U_t^i - X_t^i|^2]  \leq \frac{C}{n}\EE\int_0^T |\beta_t^1 - \bar\alpha_t^{n,i}|^2\ud t, \quad 2 \leq i \leq n,
\end{align}
and boundedness of moments of $\beta^1$ and $\bar\alpha^{n,i}$.

\medskip
\noindent\textbf{Step 2.2: MFG using $\beta^1$ \emph{vs.} N-player game using $(\beta^1, \bar\alpha^{n,2}, \ldots, \bar\alpha^{n,n})$.}
To compare $J^{n,1}(\beta^1, \bar\alpha^{n,2}, \ldots, \bar\alpha^{n,n})$ with the mean field cost $J$ given in \eqref{def:mfgJ}, we define the process $\bar U_t^1$ associated with the mean-field flow $\mu_t = \MCL({X_t})$ and strategy $\beta^1$:
\begin{equation}
    \ud \bar U_t^1 = b(t, \bar U_t^1, \mu_t, \beta_t^1) \ud t + \sigma \ud W_t^1, \quad 0 \leq t \leq T.
\end{equation}
Comparing it with $U_t^1$, and using the boundedness of $b_1$, Assumption~\ref{assu_MFG} \eqref{assu_b}, the estimate of $\EE \dF^2(\hat\nu_t^n, \mu_t)$ and Gronwall's inequality, we deduce
\begin{equation}
    \sup_{0 \leq t \leq T} \EE[|U_t^1 - \bar U_t^1|^2 ] \leq \frac{C}{n}.
\end{equation}
Therefore, a similar derivation as in Step 1 gives (replacing $X$ by $U$ and $\bar \nu$ by $\hat \nu$)
\begin{equation}
    \abs{J(\beta^1) - J^{n,1}(\beta^1, \bar\alpha^{2,n}, \ldots, \bar\alpha^{n,n})} \leq \frac{C}{\sqrt n},
\end{equation}
where $J(\beta^1)$ is the mean-field cost using $\beta^1$:
\begin{equation}
    J(\beta^1) = \EE\left[g(\bar U_T^1, \mu_T) + \int_0^T f(t, \bar U_t^1, \mu_t, \beta_t^1) \ud t\right].
\end{equation}
As $J$ is the optimal cost of the mean-field game, any strategy $\beta^1$ other than $\hat\alpha(t, X_t, \mu_t, u(t, X_t))$ will produce a higher cost, {\it i.e.}, $J(\beta^1) \geq J$. Therefore, one has

\begin{equation}\label{eq:mfgvsbeta}
  J^{n,i}(\beta^1, \bar\alpha^{2,n}, \ldots, \bar\alpha^{n,n})  \geq J(\beta^1) - \frac{C}{\sqrt n} \geq J -\frac{C}{\sqrt n}.
\end{equation}
Combining \eqref{eq:mfgvsalpha} and \eqref{eq:mfgvsbeta} gives the desired result \eqref{def_approxNash}.
\end{proof}

\section{Conclusion}
A new class of metrics, in the form of integral probability metrics, is proposed in this paper to study the convergence of empirical measures in high-dimensional spaces. We generalize the standard definition of maximum mean discrepancy by imposing specific criteria for selecting the test function space to guarantee the property of being free of the CoD. Examples of test function spaces include reproducing kernel Hilbert spaces, Barron function space, and flow-induced function spaces. Under the proposed metrics, we can show the following three cases of convergence are dimension-free: 1. The convergence of empirical measure drawn from a given distribution; 2. The convergence of $n$-particle system to the solution to McKean-Vlasov stochastic equation; 3. The construction of an $\varepsilon$-Nash equilibrium for a homogeneous $n$-player game by its mean-field limit. We also generalize the results in \cite{yang2020generalization,yang2022generalization} and show that, given a distribution close to the target distribution measured by the newly proposed metric and a certain representation of the target distribution, we can generate a distribution close to the target one in terms of the Wasserstein metric and relative entropy.

As future work, we shall deepen the study of our metric by investigating the mean-field limit of the $n$-player stochastic differential games in high dimensions, whose Nash equilibria can be given by the deep fictitious theory and algorithms \cite{Hu2:19,HaHu:19,han2020convergence,Xuanetal:21}.
Besides, we are interested in developing a similar theory ({\it cf.} Theorem~\ref{thm:regularity_improving}) for models other than the bias potential type and density type, for instance, the generative adversarial network models, which are useful and important in the machine learning community. We also plan to apply Theorem \ref{thm:regularity_improving} to the solution of the McKean-Vlasov SDE \eqref{MV-SDE} to study whether one can construct a distribution based on the solution of the $n$-particle system \eqref{n-body SDE}, which is close to the distribution of the solution of McKean-Vlasov SDE \eqref{MV-SDE} in the sense of the Wasserstein metric, total variation distance, or relative entropy. The technical difficulty therein is to analyze when the distribution of \eqref{MV-SDE} satisfies a bias potential model or density model. Finally, it is of interest to apply this class of metrics to analyze the convergence rates in other problems, for example, McKean-Vlasov models, MFGs, and stochastic gradient descent for two-layer neural networks or multi-layer neural networks \cite{chizat2018global,mei2018mean,sirignano2020mean}.

\bibliographystyle{plain}
\bibliography{Reference}

\appendix
\section{Proof of Theorem~\ref{thm: prior_estimation}}\label{app:prior_estimation}

The first part of claim (a) is from Theorem 3.2.2 in \cite{zhang2017backward}. Claim (b) follows from Theorem 5.1 and 5.5 in \cite{delarue2020master}. 

Below we prove the second part of claim (a). We denote by $C$ a generic constant that only depends on $K$ and $T$,  whose
value may change from line to line when there is no need to distinguish.
We first observe
\begin{align}
    \int_{C([0,T];\bR^d)}[\Delta(x,h)]^2\rmd \mu(x) &= \EE \left[\sup_{s,t \in [0,T], |s-t|\le h}\Bigl\|\int_{s}^t B(u,X_u)\rmd u + \int_{s}^t\Sigma(u,X_u)\rmd W_u\Bigr\|^2\right] \notag \\
    &\le Ch^2 \EE\left[1 + \sup_{0\le t \le T}\|X_t\|^2\right] + C\EE \sup_{s,t\in[0,T],|s-t|\le h}\Bigl\|\int_{s}^t\Sigma(u,X_u)\rmd W_u\Bigr\|^2 \notag \\
    \label{error_decomposition}
    & \le Ch^2 + C\EE \sup_{s,t\in[0,T],|s-t|\le h}\Bigl\|\int_{s}^t\Sigma(u,X_u)\rmd W_u\Bigr\|^2.
\end{align}
Thus, it suffices to estimate the second term in \eqref{error_decomposition}. In the sequel, we will use short notations $\Sigma_t := \Sigma(t,X_t)$ and $Y_t := \int_{0}^t\Sigma_u\rmd W_u$. 

We first work on the case $\eta = \delta_{x_0}$ for a fixed $x_0 \in \bR^d$. Fixing $s \in [0,T]$, one has
\begin{equation}
    \rmd \|Y_t - Y_s\|^2 = 2(Y_t - Y_s)\transpose\Sigma_t \rmd W_t + \|\Sigma_t\|_F^2 \rmd t.
\end{equation}
Hence, for any $\lambda > 0$, 
    $\exp\left(\lambda[\|Y_t - Y_s\|^2 - \int_{s}^{t}\|\Sigma_u\|_F^2\rmd u] - 2\lambda^2 \int_{s}^{t}\|(Y_u -Y_s)\transpose\Sigma_u\|^2\rmd u\right)$
is a nonnegative local martingale for $t \in [s,T]$, thus a supermartingale. 
Fix $a > 0$ and let $\tau$ be a stopping time defined by
\begin{equation}
    \tau = \inf\{u \in [s,t]: \|Y_u - Y_s\| \ge a\} \wedge t, \quad \inf\{\emptyset\} = +\infty.
\end{equation}
Then
   $\EE\left[\exp\left(\lambda[\|Y_\tau - Y_s\|^2 - \int_{s}^{\tau}\|\Sigma_u\|_F^2\rmd u] - 2\lambda^2 \int_{s}^{\tau}\|(Y_u -Y_s)\transpose\Sigma_u\|^2\rmd u\right)\right] \le 1.$
Noticing that for any $u \in [s,\tau]$, we have  $\|Y_u - Y_s\| \le a$, and 
\begin{equation}
    \int_{s}^{\tau}\|(Y_u -Y_s)\transpose\Sigma_u\|^2\rmd u \le a^2 \int_{s}^{\tau}\|\Sigma_u\|_F^2\rmd u \le Ca^2(1 + \sup_{0\le u \le T}\|X_u\|^2)(t-s).
\end{equation}
Consequently, 
\begin{equation}
    \EE[\exp(\lambda[\|Y_\tau - Y_s\|^2 - C(1 + \sup_{0\le u \le T}\|X_u\|^2)(t-s)] - C\lambda^2a^2(1 + \sup_{0\le u \le T}\|X_u\|^2)(t-s)] \le 1.
\end{equation}
Now, let $S_X$ be the maximum of $X_t$ on $[0,T]$, {\it i.e.}, $S_X := \sup_{0\le u \le T}\|X_u\|$. 
For a fixed constant $M > 0$, one deduces
\begin{equation}
    \EE[\exp(\lambda\|Y_\tau - Y_s\|^2)\mathbf{1}_{S_X\le M}] \le \exp(C\lambda(1+M^2)(t-s) + C\lambda^2 a^2(1+M^2)(t-s)).
\end{equation}
Hence,
\begin{align}
    \PP(\|Y_t - Y_s\| \ge a, S_X \le M) &\le \PP(\|Y_\tau - Y_s\|\ge a, S_X \le M) \\&\le \exp(-\lambda a^2 +C\lambda(1+M^2)(t-s) + C\lambda^2 a^2(1+M^2)(t-s)).
\end{align}
Picking $\lambda = [2C(1+M^2)(t-s)]^{-1}$, we know that 
\begin{equation}
    \PP(\|Y_t - Y_s\| \ge a, S_X \le M) \le C\exp\left(-\frac{a^2}{C(1+M^2)(t-s)}\right), \quad \forall s, t \in [0,T].
\end{equation}
Therefore, for any $(t_1,s_1),\dots,(t_m,s_m) \in [0,T]\times [0,T]$, using Exercise 2.5.10 and Proposition 2.5.2 in \cite{vershynin2018high}, we obtain
\begin{equation}\label{sde_max_control}
   \EE\left[\max_{1\le i \le m}\|Y_{t_i} - Y_{s_i}\|^4 \mathbf{1}_{S_X\le M}\right] \le C\log^2(m)(1 + M^4) \max_{1\le i \le m}|t_i-s_i|^2.
\end{equation}
For any $t \in [0,T]$ and any integer $k$, define
\begin{equation}
    \Pi_k(t) = h2^{-k}\lfloor \frac{2^k t}{h}\rfloor.
\end{equation}
The continuity of $Y_t$ together with inequality \eqref{sde_max_control} gives
\begin{align}
    \left(\EE\Bigl[\sup_{s,t \in [0,T],|s-t|\le h}\|Y_t - Y_s\|^4 \mathbf{1}_{S_X \le M}\Bigr]\right)^{\frac{1}{4}} 
    &\le C \left(\EE\Bigl[\sup_{t \in [0,T]}\|Y_t - Y_{\Pi_0(t)}\|^4 \mathbf{1}_{S_X\le M}\Bigr]\right)^{\frac{1}{4}} \\
    &\le C\sum_{k = 0}^{\infty} \left(\EE\Bigl[\sup_{t \in [0,T]}\|Y_{\Pi_{k+1}(t)} - Y_{\Pi_k(t)}\|^4\mathbf{1}_{S_X \le M}\Bigr]\right)^{\frac{1}{4}} \\
    &\le C(1 + M)\sum_{k=0}^{\infty}\left(\frac{h^2}{4^k}\log^2\Bigl(\frac{2^{k+2}T}{h}\Bigr)\right)^{\frac{1}{4}}\\
    &\le C(1 + M)\sqrt{h\log\left(\frac{2T}{h}\right)}.
\end{align}
Then by the Cauchy-Schwarz inequality, one has
\begin{align}
    \left(\EE\Bigl[\sup_{s,t \in [0,T],|s-t|\le h}\|Y_t - Y_s\|^2\Bigr]\right)^{\frac{1}{2}} &\le \sum_{k=1}^{+\infty}\left(\EE\Bigl[\sup_{s,t \in [0,T],|s-t|\le h}\|Y_t - Y_s\|^2\mathrm{1}_{k-1\le S_X\le k}\Bigr]\right)^{\frac{1}{2}} \\
    &\le \sum_{k=1}^{+\infty}\left(\EE\Bigl[\sup_{s,t \in [0,T],|s-t|\le h}\|Y_t - Y_s\|^4\mathrm{1}_{S_X \le k}\Bigr] \PP\Bigl(S_X \ge k-1\Bigr)\right)^{\frac{1}{4}} \\
    &\le C\sqrt{h\log\left(\frac{2T}{h}\right)}\sum_{k=1}^{+\infty}  (k+1)\PP^{\frac{1}{4}}(S_X \ge k-1).
\end{align}
Using claim (b) and Theorem 5.1 in \cite{delarue2020master}, we know that 
\begin{equation}
    \PP(\sup_{0\le u \le T}\|X_u\| \ge n)  \le 2\exp\left(-\frac{n^2}{C(1+\|x_0\|^2)}\right).
\end{equation}
Therefore,
\begin{align}
  \sum_{k=1}^{+\infty}  (k+1)\PP^{\frac{1}{4}}(S_X \ge k-1) &\le C + C\sum_{k=1}^{+\infty}(k+2)\exp\left(-\frac{k^2}{C(1+\|x_0\|^2)}\right) \\
  &\le C\Bigl[1 + \int_{0}^{+\infty}(a+2)\exp\left(-\frac{a^2}{C(1+\|x_0\|^2)}\right)\rmd a\Bigr] \le C(1+\|x_0\|),
\end{align}
which means
\begin{equation}
    \left(\EE\Bigl[\sup_{s,t \in [0,T],|s-t|\le h}\|Y_t - Y_s\|^2\Bigr]\right)^{\frac{1}{2}} \le C(1+\|x_0\|)\sqrt{h\log\left(\frac{2T}{h}\right)}.
\end{equation}
For general $\eta$, we use the above inequality to deduce
\begin{align}
    \EE\Bigl[\sup_{s,t \in [0,T],|s-t|\le h}\|Y_t - Y_s\|^2\Bigr] &\le \EE\left(\EE\Bigl[\sup_{s,t \in [0,T],|s-t|\le h}\|Y_t - Y_s\|^2\;\Bigl\vert\;X_0 = x_0\Bigr]\right) \\&\le C(1+ \EE\|\eta\|^2)h \log\left(\frac{2T}{h}\right)\le Ch\log\left(\frac{2T}{h}\right).
\end{align}
With inequality \eqref{error_decomposition}, we obtain the desired result.

\section{Discussion on Theorem~\ref{thm: process}}\label{app:functionclass}
We show in this section that, under a slightly stronger condition \eqref{assump:modulus} compared to \eqref{eq:module}, the  logarithm term in $\phi(n)$ defined in Theorem \ref{thm: process} can be removed, for all the examples of the test function classes discussed in Section \ref{sec:testfunction}.

Throughout this appendix, we assume $\mu$ to be a distribution on $C([0,T];\bR^d)$ such that
\begin{equation}
     \int_{C([0,T];\bR^d)}\sup_{0\le t \le T}\|x_t\|^2\rmd \mu(x) < +\infty,
\end{equation}
and
\begin{equation}\label{assump:modulus}
     \int_{C([0,T];\bR^d)}[\Delta(x,h)]^2\rmd \mu(x) \le Q h^\alpha \log^\beta\left(\frac{2T}{h}\right),
\end{equation}
for any $h > 0$, where $Q,\alpha,\beta > 0$ are positive constants. Still, we will use $C$ to denote a positive constant depending only on $\alpha$ and $\beta$, which may vary from line to line.

The first result is established for the reproducing kernel Hilbert spaces (RKHSs).
\begin{prop}\label{rkhs_without_log}
Assume the kernel k satisfies the condition (a) in Theorem \ref{thm: rkhs i.i.d.} and $\Phi$ is the unit ball of $\mathcal{H}_k$, then
\begin{equation}
   \EE \sup_{t\in[0,T]} \rmD_\Phi(\mu_t,\bar{\mu}_t^n)\le 2\sqrt{\frac{2}{n}\Big[K_1^2\int_{C([0,T];\bR^d)}\sup_{0\le t \le T}\|x_t\|^2\rmd\mu(x) + K_2^2\Big]} + CK_1\sqrt{\frac{QT^\alpha}{n}}.
\end{equation}
\end{prop}
\begin{proof}

 Let $X^1,\dots,X^n$ be i.i.d. processes drawn from the distribution $\mu$. Following \cite[Lemma 26.2]{shalev2014understanding}, we immediately have
 \begin{equation}
 \EE\sup_{t\in[0,T]}\rmD_\Phi(\mu_t,\bar{\mu}_t^n) \le \frac{2}{n}\EE\sup_{t \in [0,T]}\sup_{\|f\|_{\mathcal{H}_k} \le 1}\Big|\sum_{i=1}^n\xi_if(X_t^i)\Big|. 
 \end{equation}
For a fixed integer $n'\ge 2$ and any $t_p,s_p \in [0,T]$,  $p = 1,\dots,n'$, we first compute
\begin{align}
      &\frac{1}{n}\EE \max_{1\le p \le n'}\sup_{\|f\|_{\mathcal{H}_k} \le 1} \Big|\sum_{i=1}^n\xi_i [f(X^i_{t_p})-f(X^i_{s_p})]\Big|\notag \\
    =&\frac{1}{n}\EE \max_{1\le p \le n'}\sup_{\|f\|_{\mathcal{H}_k} \le 1} \Big|\langle f,\sum_{i=1}^n\xi_i (k(X_{t_p}^i,\cdot) - k(X_{s_p}^i,\cdot)\rangle_{\mathcal{H}_k}\Big| \notag \\
     \leq & \frac{1}{n}\EE\max_{1\le p \le n'} \sqrt{\sum_{i=1}^n\sum_{j=1}^n\xi_i\xi_j[k(X_{t_p}^i,X_{t_p}^j) + k(X_{s_p}^i,X_{s_p}^j)-k(X_{t_p}^i,X_{s_p}^j) - k(X_{s_p}^i,X_{t_p}^j)]}. \label{eq:rkhsapp:eq2}
 \end{align}
Let $S$  be a nonnegative-definite matrix, then there exists a nonnegative-definite $n\times n$ matrix $F$ such that $S = F\transpose F$. Let $\xi = (\xi_1,\dots,\xi_n)\transpose$, then there exists a universal constant $C > 0$ ({\it cf.}~\cite[Theorem~2.1]{rudelson2013hanson} and \cite[Example~2.5.8]{vershynin2018high}), such that for any $a \ge 0$,
 \begin{equation}
     \mathbb{P}(\sqrt{\xi\transpose S \xi} - \sqrt{\mathrm{Trace}(S)} \ge a) = \mathbb{P}(\|F\xi\| - \|F\|_{\mathrm{HS}} \ge a) \le 2\exp(-\frac{a^2}{C \mathrm{Trace}(S)}), \label{eq:rkhsapp:eq3}
 \end{equation}
 where $\|F\|_{\mathrm{HS}} = \sqrt{\mathrm{Trace}(F\transpose F)}$.
Therefore, for any positive definite matrices $S_1,\dots,S_{n'}$,  \cite[Exercise~2.5.10]{vershynin2018high} gives
\begin{equation}
    \EE \sup_{1 \le p \le n'} \sqrt{\xi\transpose S_p \xi} \le C\max_{1\le p \le n'}\sqrt{\mathrm{Trace}(S_p)\log n'}. 
\end{equation}
Since for any $a_1,a_2,\dots,a_n \in \bR$,
 \begin{align}
     &\sum_{i=1}^n\sum_{j=1}^na_ia_j[k(X_{t_p}^i,X_{t_p}^j) + k(X_{s_p}^i,X_{s_p}^j)-k(X_{t_p}^i,X_{s_p}^j) - k(X_{s_p}^i,X_{t_p}^j)]\notag \\
    =&\langle \sum_{i=1}^na_i[k(X_{t_p}^i,\cdot)-k(X_{s_p}^i,\cdot)], \sum_{i=1}^na_i[k(X_{t_p}^i,\cdot)-k(X_{s_p}^i,\cdot)]\rangle_{\mathcal{H}_k} \ge 0,
 \end{align}
using \eqref{eq:rkhsapp:eq2} and  \eqref{eq:rkhsapp:eq3} with $(K_p)_{i,j} = k(X_{t_p}^i,X_{t_p}^j) + k(X_{s_p}^i,X_{s_p}^j)-k(X_{t_p}^i,X_{s_p}^j) - k(X_{s_p}^i,X_{t_p}^j)$, we deduce
\begin{align}
     &\frac{1}{n}\EE \max_{1\le p \le n'}\sup_{\|f\|_{\mathcal{H}_k} \le 1} \Big|\sum_{i=1}^n\xi_i [f(X^i_{t_p})-f(X^i_{s_p})]\Big|\notag \\ 
   &\le \frac{C}{n}\sqrt{\log n'}\EE\max_{1\le p \le n'}\sqrt{\sum_{i=1}^n k(X_{t_p}^i,X_{t_p}^i) + k(X_{s_p}^i,X_{s_p}^i) - 2k(X_{t_p}^i,X_{s_p}^i)}\notag \\
     &\le C \sqrt{\frac{\log n'}{n}\EE\max_{1\le p \le n'}[k(X_{t_p}^1,X_{t_p}^1) + k(X_{s_p}^1,X_{s_p}^1) - 2k(X_{t_p}^1,X_{s_p}^1)]}\notag \\
     &\le  CK_1\sqrt{\frac{\log n'}{n}\EE \max_{1\le p\le n'}\|X_{t_p}^1 - X_{s_p}^1\|^2}, \label{eq:rkhsapp:eq1}
\end{align}
where $K_1$ is the constant defined in Theorem~\ref{thm: rkhs i.i.d.}.

We now use the chaining method to estimate the Rademacher complexity. For any $t \in [0,T]$ and integer $k$, we define
\begin{equation}
    \Pi_k(t) = T2^{-k}\lfloor\frac{2^k t}{T}\rfloor.
 \end{equation}
We first compute
\begin{align}
    &\Bigg|\frac{1}{n}\EE\sup_{t \in [0,T]}\sup_{\|f\|_{\mathcal{H}_k} \le 1}\Big|\sum_{i=1}^n\xi_if(X_t^i)\Big| - \frac{1}{n}\EE\sup_{t \in [0,T]}\sup_{\|f\|_{\mathcal{H}_k} \le 1}\Big|\sum_{i=1}^n\xi_i f(X_{\Pi(t)}^i)\Big|\Bigg| \notag \\
    \le&\frac{1}{n}\EE\sup_{t \in [0,T]}\sup_{\|f\|_{\mathcal{H}_k} \le 1}\Big|\sum_{i=1}^n \xi_i [f(X_t^i) - f(X_{\Pi(t)}^i)]\Big| \notag \\
     \le& K_1\EE \sup_{t \in [0,T]}\|X_t^1 - X_{\Pi(t)}^1\| \le CK_1 Q\left(\frac{T}{2^k}\right)^{\alpha/2} (k+1)^{\beta/2} \rightarrow 0,
\end{align} 
as $k \rightarrow +\infty$. We have also derived in Theorem~\ref{thm: rkhs i.i.d.} that
 \begin{equation}
     \frac{1}{n}\EE\sup_{\|f\|_{\mathcal{H}_k} \le 1}|\sum_{i=1}^n\xi_if(X_0^i)| \le \sqrt{\frac{2}{n}\Big[K_1^2\int_{C([0,T];\bR^d)}\|x_0\|^2\rmd\mu(x) + K_2^2\Big]}.
 \end{equation}
Together with \eqref{eq:rkhsapp:eq1}, we finally achieve
\begin{align}
 &\frac{1}{n}\EE\sup_{t \in [0,T]}\sup_{\|f\|_{\mathcal{H}_k} \le 1}\Big|\sum_{i=1}^n\xi_if(X_t^i)\Big| \notag \\
     \le& \frac{1}{n}\EE\sup_{\|f\|_{\mathcal{H}_k} \le 1}\Big|\sum_{i=1}^n\xi_if(X_0^i)\Big| + \sum_{k=0}^{+\infty}\frac{1}{n}\EE\sup_{t \in [0,T]}\sup_{\|f\|_{\mathcal{H}_k \le 1}}\Big|\sum_{i=1}^n\xi_i [f(X_{\Pi_{k+1}(t)}^i)- f(X_{\Pi_k(t)}^i)]\Big| \notag \\
   \le&\sqrt{\frac{2}{n}\Big[K_1^2\int_{C([0,T];\bR^d)}\|x_0\|^2\rmd\mu(x) + K_2^2\Big]} + CK_1\sum_{k=0}^{+\infty}\sqrt{\frac{\log (2^{k+1})}{n}\EE \sup_{t \in [0,T]}\|X_{\Pi_{k+1}(t)}^1 - X_{\Pi_k(t)}^1\|^2} \notag \\
    \le& \sqrt{\frac{2}{n}\Big[K_1^2\int_{C([0,T];\bR^d)}\sup_{0\le t \le T}\|x_t\|^2\rmd\mu(x) + K_2^2\Big]} + CK_1 \sum_{k=0}^{+\infty}\sqrt{\frac{k+1}{n}Q\left(\frac{T}{2^k}\right)^\alpha (k+1)^\beta} \notag \\
     \le&\sqrt{\frac{2}{n}\Big[K_1^2\int_{C([0,T];\bR^d)}\sup_{0\le t \le T}\|x_t\|^2\rmd\mu(x) + K_2^2\Big]} + CK_1\sqrt{\frac{QT^\alpha}{n}},
\end{align}
where we have used $\sum_{k=0}^{+\infty}\sqrt{\frac{(k+1)^{\beta+1}}{2^{\alpha k}}} < +\infty$.
\end{proof}

To establish results for the Barron space and flow-induced function spaces, we first present the following lemma.
\begin{lem}\label{lem_estimate_sup}
Let $X^1,\dots,X^n$ be i.i.d. processes drawn from $\mu \in \mc{P}^2(C([0,T]; \RR^d))$ and $\xi_1,\dots,\xi_n$ be i.i.d. Rademacher variables which are independent of $X^1,\dots,X^n$. Then,
\begin{equation}
    \frac{1}{n}\EE \sup_{0\le t \le T}\sqrt{\sum_{i=1}^n\sum_{j=1}^n\xi_i\xi_j 
    [(X_t^i)\transpose X_t^j+1]} \le \sqrt{\frac{2}{n}\int_{C([0,T];\bR^d)}\sup_{0\le t \le T}[\|x_t\|+1]^2\rmd\mu(x)} + C\sqrt{\frac{QT^\alpha}{n}}.
\end{equation}
\end{lem}
\begin{proof}
Taking $k(x,x') = x\transpose x' + 1$ for any $x,x' \in \bR^d$, this lemma can be derived using the proof of Proposition \ref{rkhs_without_log}.
\end{proof}
\begin{prop}
\begin{enumerate}
    \item Let $\Phi = \mathcal{B}_1$ be the unit ball of Barron space $\mathcal{B}$, then

\begin{equation}
   \EE \sup_{0\le t \le T} \rmD_\Phi(\mu_t,\bar{\mu}_t^n)\le 4\sqrt{\frac{2}{n}\int_{C([0,T];\bR^d)}\sup_{0\le t \le T}[\|x_t\|^2+1]\rmd\mu(x)} + C\sqrt{\frac{QT^\alpha}{n}}.
\end{equation}
   \item Let $\Phi = \{f \in \mathcal{D}, \|f\|_{\mathcal{D}} \leq 1 \}$ be the unit ball of flow-induced function spaces $\mathcal{D}$, then
   \begin{equation}
        \EE \sup_{0\le t \le T} \rmD_\Phi(\mu_t,\bar{\mu}_t^n)\le 2e^2\sqrt{\frac{2}{n}\int_{C([0,T];\bR^d)}\sup_{0\le t \le T}[\|x_t\|^2+1]\rmd\mu(x)} + C\sqrt{\frac{QT^\alpha}{n}}.
   \end{equation}
   \end{enumerate}
\end{prop}
\begin{proof}
The proof of these arguments is quite similar with the proof of claim (f) in Theorem \ref{thm: Property_Barron_Space} and claim (c) in Theorem \ref{thm: compositional} with the above Lemma \ref{lem_estimate_sup}.

With \cite[Lemma 26.2]{shalev2014understanding}, and following the proofs of claim (f) in Theorem \ref{thm: Property_Barron_Space} and claim (c) in Theorem \ref{thm: compositional}, we obtain
\begin{equation}
     \EE \sup_{0\le t \le T} \rmD_\Phi(\mu_t,\bar{\mu}_t^n) \le K_\Phi \frac{1}{n}\EE \sup_{0\le t \le T}\sqrt{\sum_{i=1}^n\sum_{j=1}^n\xi_i\xi_j 
    [(X_t^i)\transpose X_t^j+1]}, 
\end{equation}
where $K_\Phi = 4$ in case 1, and $K_\Phi = 2e^2$ in case 2. Then we conclude our results by applying Lemma~\ref{lem_estimate_sup}.
\end{proof}

\end{document}